\documentclass[10pt]{amsart}
\usepackage{indentfirst,enumerate,amssymb,amsfonts,amsmath,amsthm,mathrsfs,dsfont}
\usepackage[colorlinks=true,linkcolor=blue,citecolor=blue]{hyperref}

\numberwithin{equation}{section}
\theoremstyle{plain}
\newtheorem{thm}{Theorem}[section]
\newtheorem{prop}[thm]{Proposition}
\newtheorem{cor}[thm]{Corollary}
\newtheorem{lemma}[thm]{Lemma}
\newtheorem{defi}[thm]{Definition}
\newtheorem{rem}[thm]{Remark}
\newtheorem{ex}[thm]{Example}
\newtheorem{conjecture}[thm]{Conjecture}
\newtheorem{question}{Question}[section]

\newcommand{\Rep}{\mbox{Rep}}
\newcommand{\HS}{{\mathtt{HS}}}
\newcommand{\Tr}{\mbox{\emph{Tr}}}
\newcommand{\N}{{\mathbb{N}}}
\newcommand{\Q}{\mathbb{Q}}
\newcommand{\R}{\mathbb{R}}
\newcommand{\Z}{\mathbb{Z}}
\newcommand{\C}{\mathbb{C}}

\newcommand{\St}{{\mathbb S}^3}
\newcommand{\Sth}{\widehat{{\mathbb S}^3}}
\newcommand{\DG}{\mathcal{D}'(G)}
\newcommand{\TS}{\mathbb{T}^1\times\St}

\newcommand{\jp}[1]{{\left\langle{#1}\right\rangle}}

\newlength{\dhatheight}
\newcommand{\doublehat}[1]{%
	\settoheight{\dhatheight}{\ensuremath{\hat{#1}}}
	\addtolength{\dhatheight}{-0.15ex}
	\widehat{\vphantom{\rule{5pt}{\dhatheight}}%
		\smash{\widehat{#1}}}}

\begin{document}

%
%
%
%
%
%
%
%
%

\title[Vector fields on compact Lie groups]{Global hypoellipticity and global solvability \\ for vector fields on compact Lie groups}


\author[A. Kirilov]{Alexandre Kirilov}
\address{
	Universidade Federal do Paran\'{a}, 
	Departamento de Matem\'{a}tica,
	C.P.19096, CEP 81531-990, Curitiba, Brazil
}
\email{akirilov@ufpr.br}


\author[W. A. A. de Moraes]{Wagner A. A. de Moraes}
\address{
	Universidade Federal do Paran\'{a},
	Programa de P\'os-Gradua\c c\~ao de Matem\'{a}tica,\linebreak
	C.P.19096, CEP 81531-990, Curitiba, Brazil
}

\email{wagneramat@gmail.com}


\author[M. Ruzhansky]{Michael Ruzhansky}
\address{Ghent University, 
	Department of Mathematics: Analysis, Logic and Discrete Mathematics, 
	Ghent, Belgium 
	and 
	Queen Mary University of London, 
	School of Mathematical Sciences, 
	London, United Kingdom
}
\thanks{This study was financed in part by the Coordenação de Aperfeiçoamento de Pessoal de Nível Superior - Brasil (CAPES) - Finance Code 001. The last author was also supported by the FWO Odysseus grant, by the Leverhulme Grant RPG-2017-151, and by EPSRC Grant EP/R003025/1.}

\email{Michael.Ruzhansky@ugent.be}




\begin{abstract}
 We present necessary and sufficient conditions to have global hypoellipticity and global solvability for a class of vector fields defined on a product of compact Lie groups. In view of Greenfield's and Wallach's conjecture, about the non-existence of globally hypoelliptic vector fields on compact manifolds different from tori, we also investigate different notions of regularity weaker than global hypoellipticity and describe completely the global hypoellipticity and global solvability of zero-order perturbations of our vector fields. We also present a class of vector fields with variable coefficients whose operators can be reduced to a normal form, and we prove that the study of the global properties of such operators is equivalent to the study of the respective properties for their normal forms.
\end{abstract}
\raggedbottom
\maketitle
\tableofcontents
\raggedbottom
\section{Introduction}
In this note, we study the regularity of solutions and solvability of vector fields (and their perturbations by zero order terms) on a compact Lie group $G$. More precisely, if $\mathcal{D}'(G)$ stands for the space of distributions on $G$ and ${P:\mathcal{D}'(G) \rightarrow \mathcal{D}'(G)}$ is a first-order differential operator, we are interested in establishing conditions that ensure that $u$ is smooth whenever $Pu$ is smooth. We  also want to identify under what conditions it is possible to guarantee that $Pu=f \in \mathcal{D}'(G)$ has a solution, in the sense of distributions. These properties are known as global hypoellipticity and global solvability, and have been widely studied in recent years, especially on the $d-$dimensional torus $\mathbb{T}^d$. See, for example, the impressive list of authors who have published articles addressing these subjects: \cite{Ber99}, \cite{BerCorMal93}, \cite{BCP04}, \cite{CC00}, \cite {AGKM18}, \cite{Gon18}, \cite{GPY92}, \cite{GW72}, \cite{GW73a}, \cite{GW73b}, \cite{H01}, \cite{HP04}, \cite{HP06}, \cite{Hou79}, \cite{Hou82}, \cite{Pet11}, \cite{PetrZan08} and references therein.

Even in the case of $\mathbb{T}^d$, the investigation of these properties for vector fields is a challenging problem that still has open questions. Perhaps, the most famous and seemingly far-off question of a solution is the Greenfield's and Wallach's conjecture, which states the following: if a closed smooth orientable manifold admits a globally hypoelliptic vector field, then this manifold is $C^\infty-$diffeomorphic to a torus and this vector field is $C^\infty-$conjugated to a constant vector field whose coefficients satisfy a diophantine condition (see \cite {F08} and \cite {GW73a}).

Most of the studies that deal with the question of global hypoellipticity and global solvability in the torus make use of the Fourier analysis as the main tool to obtain results from conditions imposed on the symbol or on the coefficients of the operator. For example, in \cite{GW72}, S. Greenfield and N. Wallach use only the Fourier series in $\mathbb{T}^d$ to characterize the global hypoellipticity of a differential operator through its symbol. In this paper there appears for the first time the famous application: $L=\partial_x+a\partial_y,$ $a \in \mathbb{R}$ is globally hypoelliptic in $\mathbb{T}^2$ if, and only if, $a$ is an irrational non-Liouville number. Therefore a natural way of extending such studies to other smooth manifolds would be to consider manifolds where we have a Fourier analysis.

In this direction, based on ideas \cite{GW73b} and \cite{See65}, J. Delgado and the third author \cite{DR18b} introduced on compact smooth manifolds $M$ a notion of Fourier series for operators that commute with a fixed elliptic operator. Using these ideas, a study of global hypoellipticity for such operators was made in \cite{AGK18}, \cite{AK19}, and \cite{KM}. The obvious disadvantage of this technique is that for now, it works only for operators that commute with a fixed elliptic operator.

In the particular case where the compact manifold is a Lie group $G$, there is a natural way of introducing a Fourier analysis into $G$, see for example \cite{DR14}, \cite{DR16}, \cite{DR18}, \cite{FR16}, \cite{RT10}, \cite{RT13},  \cite{RTW14}, \cite{RW13}, \cite{Tay68}. In this paper we use the notation and results based on the book by M. Ruzhansky and V. Turunen \cite{livropseudo} to study the global hypoellipticity and global solvability of a class of vector fields defined on Lie groups.

In the development of this project, we find natural to begin by extending the results of \cite{GW72} to a product of Lie groups $G_1\times G_2$. More precisely, if $L=X_1 + a X_2$, where $X_j$ is a vector field on the Lie algebra $\mathfrak{g}_j$ and $a \in \C$, then we characterize completely the global hypoellipticity and global solvability of $L$, presenting necessary and sufficient conditions on the behavior of its symbol $\sigma_L$. In a subsection dedicated to examples, we recover the results obtained in \cite{GW72} to the torus and present an example on $\TS$. This case was called ``constant-coefficient vector field''. 

Given the probable validity of the Greenfield's and Wallach's conjecture, we introduce two different notions of global regularity weaker than the global hypoellipticity. The first one,  global hypoellipticity modulo kernel, was inspired by the paper \cite{Ara19} of G. Ara\'ujo; and the second, $\mathcal{W}$--global hypoellipticity, emerged naturally in the development of this study. In both cases, we also characterize the regularity of $L=X_1 + a X_2$ by analyzing the behavior of the symbol $\sigma_L$.  To complement the study of the case of constant coefficients vector fields, we consider perturbations of $L$ by low order terms, obtaining necessary and sufficient conditions to have global hypoellipticity and global solvability.

Finally, we introduce the class of variable coefficients vector fields of the form $ L = X_1 + a (x_1) X_2 $, where $a \in C^\infty(G_1)$ is a real-valued function. We show that the vector field of this class can be reduced to a normal form, so the study of the global properties of such operators is equivalent to the study of the respective properties of their normal forms.

The paper is organized as follows. In Section 2 we recall some classical results about Fourier analysis on compact Lie groups and fix the notation that will be used throughout the text. In Section 3 we give necessary and sufficient conditions for the global hypoellipticity and the global solvability of constant-coefficient vector fields defined on compact Lie groups and we present a class of examples. Because of the Greenfield-Wallach conjecture, in Section 4 we present two weaker notions of hypoellipticity. In Section 5 we study perturbations of vector fields by low order terms, both by constants and by functions. Finally, in Section 6 we characterize global properties of a class of perturbed vector fields with variable coefficients. We present the normal form and establish a connection between the global properties of perturbed vector fields with variable and constant coefficients.

\raggedbottom
	\section{Fourier analysis on compact Lie groups}
In this section we recall most of the notations and preliminary results necessary for the development of this study. A very careful presentation of these concepts and the demonstration of all the results presented here can be found in the references \cite{FR16}  and \cite{livropseudo}.

Let $G$ be a compact Lie group and let $\Rep(G)$  be the set of continuous irreducible unitary representations of $G$. Since $G$ is compact, every continuous irreducible unitary representation $\phi$ is finite dimensional and it can be viewed as a matrix-valued function $\phi: G \to \C^{d_\phi\times d_\phi}$, where $d_\phi = \dim \phi$. We say that $\phi \sim \psi$ if there exists an unitary matrix $A\in C^{d_\phi \times d_\phi}$ such that $A\phi(x) =\psi(x)A$, for all $x\in G$. We will denote by $\widehat{G}$ the quotient of $\Rep(G)$ by this equivalence relation.

For $f \in L^1(G)$ the group Fourier transform of $f$ at $\phi \in \Rep(G)$ is
\begin{equation*}
\widehat{f}(\phi)=\int_G f(x) \phi(x)^* \, dx,
\end{equation*}
where $dx$ is the normalized Haar measure on $G$.
By the Peter-Wyel theorem, we have that 
\begin{equation}\label{ortho}
\mathcal{B} := \left\{\sqrt{\dim \phi} \, \phi_{ij} \,; \ \phi=(\phi_{ij})_{i,j=1}^{d_\phi}, [\phi] \in \widehat{G} \right\},
\end{equation}
is an orthonormal basis for $L^2(G)$, where we pick only one matrix unitary representation in each class of equivalence, and we may write
\begin{equation*}
f(x)=\sum_{[\phi]\in \widehat{G}}d_\phi \emph{\Tr}(\phi(x)\widehat{f}(\phi)).
\end{equation*}
Moreover, the Plancherel formula holds:
\begin{equation}
\label{plancherel} \|f\|_{L^{2}(G)}=\left(\sum_{[\phi] \in \widehat{G}}  d_\phi \ 
\|\widehat{f}(\phi)\|_{\HS}^{2}\right)^{\tfrac{1}{2}}=:
\|\widehat{f}\|_{\ell^{2}(\widehat{G})},
\end{equation}
where 
\begin{equation*} \|\widehat{f}(\phi)\|_{\HS}^{2}=\emph{\Tr}(\widehat{f}(\phi)\widehat{f}(\phi)^{*})=\sum_{i,j=1}^{d_\phi}  \bigr|\widehat{f}(\phi)_{ij}\bigr|^2.
\end{equation*}
Let $\mathcal{L}_G$ be the Laplace-Beltrami operator of $G$. For each $[\phi] \in \widehat{G}$, its matrix elements are eigenfunctions of $\mathcal{L}_G$ correspondent to the same eigenvalue that we will denote by $-\nu_{[\phi]}$, where $\nu_{[\phi]} \geq 0$. Thus
$$
-\mathcal{L}_G \phi_{ij}(x) = \nu_{[\phi]}\phi_{ij}(x), \quad \textrm{for all } 1 \leq i,j \leq d_\phi,
$$
and we will denote by
$$
\jp \phi := \left(1+\nu_{[\phi]}\right)^{1/2}
$$
the eigenvalues of $(I-\mathcal{L}_G)^{1/2}.$ We have the following estimate for the dimension of $\phi$  (Proposition 10.3.19 of \cite{livropseudo}): there exists $C>0$ such that for all $[\xi] \in \widehat{G}$ it holds
\begin{equation*}
d_\phi \leq C \jp{\phi}^{\frac{\dim G}{2}}.
\end{equation*}

For $x\in G$, $X\in \mathfrak{g}$ and $f\in C^\infty(G)$, define 
$$
L_Xf(x):=\frac{d}{dt} f(x\exp(tX))\bigg|_{t=0}.
$$

The operator $L_X$ is left-invariant, that is, $\pi_L(y)L_X = L_X\pi_L(y)$, for all $y \in G$. When there is no possibility of ambiguous meaning, we will write only $Xf$ instead of $L_Xf$. 

Let $G$ be a compact Lie group of dimension $d$ and $\{X_i\}_{i=1}^d$ a basis of its Lie algebra. For a multi-index $\alpha =(\alpha_1,\cdots \alpha_d)\in \N_0^d$,  the left-invariant differential operator of order $|\alpha|$ is
\begin{equation}\label{derivative}
\partial^\alpha := Y_1 \cdots Y_{|\alpha|},
\end{equation}
with $Y_j \in \{X_i \}_{i=1}^d$, $1 \leq j \leq |\alpha|$ and $\sum\limits_{j:Y_j=X_k}1=\alpha_k$ for every $1\leq k \leq d$. It means that $\partial^\alpha$ is a composition of left-invariant derivatives with respect to vector $X_1, \dots, X_d$ such that each $X_k$ enters $\partial^\alpha$ exactly $\alpha_k$ times. We do not specify in the notation $\partial^\alpha$ the order of vectors $X_1,\dots,X_d$, but this will not be relevant for the arguments that we will use in this article. 

We endow $C^\infty(G)$ with the usual Fréchet space topology defined by the family of seminorms $p_\alpha(f) = \max\limits_{x \in G} |\partial^\alpha f(x)|$, $\alpha \in \N$. Thus, the convergence on $C^\infty(G)$ is just the uniform convergence of functions and all their derivatives. As usual the space of distributions $\mathcal{D}'(G)$ is the space of all continuous linear functionals on $C^\infty(G)$.	
For $u \in \DG$, we define the distribution $\partial^\alpha u$ as
$$
\jp{\partial^\alpha u, \psi} := (-1)^{|\alpha|} \jp{u,\partial^\alpha \psi},
$$
for all $\psi \in C^\infty(G)$, and $\alpha \in \N_0^d$.

	Let  $P: C^{\infty}(G) \to C^{\infty}(G)$ be a continuous linear operator. The  symbol of the operator $P$ in $x\in G$ and $\phi \in \mbox{{Rep}}(G)$, $\phi=(\phi_{ij})_{i,j=1}^{d_\phi}$ is
$$
\sigma_P(x,\phi) := \phi(x)^*(P\phi)(x) \in \C^{d_\phi \times d_\phi},
$$
where $(P\phi)(x)_{ij}:= (P\phi_{ij})(x)$, for all $1\leq i,j \leq d_\phi$, and we have
$$
Pf(x) = \sum_{[\phi] \in \widehat{G}} \dim (\phi) \mbox{Tr} \left(\phi(x)\sigma_P(x,\phi)\widehat{f}(\phi)\right)
$$
for every $f \in C^\infty(G)$ and $x\in G$.

Notice that the last expression is independent of the choice of the representative. When $P: C^\infty(G) \to C^\infty(G)$ is a continuous linear left-invariant operator, that is $P\pi_L(y)=\pi_L(y)P$, for all $y\in G$, we have that $\sigma_P$ is independent of $x\in G$ and
$$
\widehat{Pf}(\phi) = \sigma_P(\phi)\widehat{f}(\phi),
$$
for all $f \in C^\infty(G)$ and $[\phi] \in \widehat{G}$.

Let $X\in \mathfrak{g}$ be a vector field normalized by the norm induced by the Killing form. It is easy to see that the operator $iX$ is symmetric on $L^2(G)$. Hence, for all $[\phi] \in \widehat{G}$ we can choose a representative $\phi$ such that $\sigma_{iX}(\phi)$ is a diagonal matrix, with entries $\lambda_m \in \R$, $1 \leq m \leq d_\phi$. By the linearity of the symbol, we obtain
$$
\sigma_X(\phi)_{mn} = i\lambda_m \delta_{mn}, \quad \lambda_j \in \R.
$$
Notice that $\{\lambda_m\}_{m=1}^{d_\phi}$ are the eigenvalues of $\sigma_{iX}(\phi)$ and then are independent of the choice of the representative, since the symbol of equivalent representations are similar matrices. Moreover, since $-(\mathcal{L}_G - X^2)$ is a positive operator and commutes with $X^2$,
$$
|\lambda_m(\phi)| \leq \jp{\phi},
$$
for all $[\phi] \in \widehat{G}$ and $1 \leq m \leq d_\phi$.

Smooth functions and distributions can be characterized by their Fourier coefficients:

\begin{prop}[Page 759 of \cite{DR14}] \label{smoo}
	Let $G$ be a compact Lie group. The following three statements are equivalent:
	\begin{enumerate}[1.] 
		\item $f \in C^\infty(G)$;
		\item for each $N>0$, there exists $C_N > 0$ such that
		$$
		\|\widehat{f}(\phi)\|_{\HS} \leq C_N \jp{\phi}^{-N},
		$$
		 for all $[\phi] \in \widehat{G}$;
		\item for each $N>0$, there exists $C_N > 0$ such that
		$$
		|\widehat{f}(\phi)_{ij}| \leq C_N \jp{\phi}^{-N}, 
		$$
for all $[\phi] \in \widehat{G}$ and $1\leq i,j \leq d_\phi$.
	\end{enumerate}
Moreover, the following three statements are equivalent:
	\begin{enumerate}[1.]\setcounter{enumi}{3}
		\item $u \in \DG$;
		\item there exist $C, \, N > 0$ such that 
		$$
		\|\widehat{u}(\phi)\|_{\HS} \leq C \jp{\phi}^{N}, 
		$$
for all $[\phi] \in \widehat{G}$;
		\item there exist $C, \, N > 0$ such that 
		$$
		|\widehat{u}(\phi)_{ij}| \leq C \jp{\phi}^{N},
		$$
		for all $[\phi] \in \widehat{G}$ and $1\leq i,j \leq d_\phi$.
	\end{enumerate}
\end{prop} 

	Let $G_1$ and $G_2$ be compact Lie groups and set $G=G_1\times G_2$. Given $f \in L^1(G)$ and  $\xi \in {\Rep}(G_1)$, the partial Fourier coefficient of $f$ with respect to the first variable  is defined by 
$$
\widehat{f}(\xi, x_2) = \int_{G_1} f(x_1,x_2)\, \xi(x_1)^* \, dx_1 \in \C^{d_\xi \times d_\xi}, \quad x_2 \in G_2,
$$
with components
$$
\widehat{f}(\xi, x_2)_{mn} = \int_{G_1} f(x_1,x_2)\, \overline{\xi(x_1)_{nm}} \, dx_1, \quad 1 \leq m,n\leq d_\xi.
$$
Analogously we define the partial Fourier coefficient of $f$ with respect to the second variable. Notice that, by definition, $\widehat{f}(\xi,\: \cdot \:)_{mn} \in C^\infty(G_2)$ and $\widehat{f}(\: \cdot \:, \eta)_{rs} \in C^\infty(G_1)$. 

Let $u \in \DG$, $\xi \in {\Rep}(G_1)$ and $1\leq m,n \leq d_\xi$. The $mn$-component of  the partial Fourier coefficient of $u$ with respect to the first variable is the linear functional defined by
$$
\begin{array}{rccl}
\widehat{u}(\xi, \: \cdot\: )_{mn}: & C^\infty(G_2) & \longrightarrow & \C \\
& \psi & \longmapsto & \jp{\widehat{u}(\xi, \: \cdot \:)_{mn},\psi} := \jp{u,\overline{\xi_{nm}}\times\psi}_G.
\end{array}
$$
In a similar way, for $\eta \in {\Rep}(G_2)$ and $1 \leq r,s\leq d_\eta$, we define the $rs$-component of the partial Fourier coefficient of $u$ with respect to the second variable.
It is easy to see that $ \widehat{u}(\xi, \: \cdot\: )_{mn} \in \mathcal{D}'(G_2)$ and $\widehat{u}( \: \cdot\:,\eta )_{rs} \in \mathcal{D}'(G_1)$. 

Notice that
\begin{equation*}
\doublehat{\,u\,}\!(\xi,\eta)_{mn_{rs}} = \doublehat{\,u\,}\!(\xi,\eta)_{rs_{mn}} =  \widehat{u}(\xi \otimes \eta)_{ij},
\end{equation*}
with $i = d_\eta(m-1)+ r$ and $j  =  d_\eta(n-1) + s$, whenever $u \in C^\infty(G)$ or $u \in \DG$.

Finally, we have the following characterization of smooth functions and distributions on $G=G_1\times G_2$:

\begin{prop}[Theorems 3.3 and 3.4 of \cite{KMR19}]\label{smoopartial}
	Let $G_1$ and $G_2$ be compact Lie groups, and set $G=G_1\times G_2$ . The following three statements are equivalent:
	\begin{enumerate}[1.]
		\item $f \in C^\infty(G)$;
		\item for each $N>0$, there exists $C_N>0$ such that
		$$
		\|\doublehat{\,f\,}\!(\xi,\eta)\|_{\HS} \leq C_N(\jp{\xi}+\jp{\eta})^{-N},
		$$
		for all $[\xi] \in \widehat{G_1}$ and $[\eta] \in \widehat{G_2}$;
		\item for each $N>0$, there exists $C_N>0$ such that
		$$
		\Big|\doublehat{\,f\,}\!(\xi,\eta)_{mn_{rs}}  \Big| \leq C_N (\jp{\xi}+\jp{\eta})^{-N},
		$$
		for all $[\xi] \in \widehat{G_1}$, $[\eta] \in \widehat{G_2}$, $1 \leq m,n \leq d_\xi$,  and $1 \leq r,s \leq d_\eta$.
	\end{enumerate}

Moreover, the following statements are equivalent:
	\begin{enumerate}[1.]\setcounter{enumi}{3}
		\item  $u \in \DG$;
		\item there exists $C, \, N >0$ such that
		$$
		\|\doublehat{\,u\,}\!(\xi,\eta)\|_{\HS} \leq C(\jp{\xi}+\jp{\eta})^{N},
		$$
		for all $[\xi] \in \widehat{G_1}$ and $[\eta] \in \widehat{G_2}$;
		\item there exists $C, \, N >0$ such that
		$$
		\Big|\doublehat{\,u\,}\!(\xi,\eta)_{mn_{rs}}  \Big| \leq C (\jp{\xi}+\jp{\eta})^{N},
		$$
		for all $[\xi] \in \widehat{G_1}$, $[\eta] \in \widehat{G_2}$, $1 \leq m,n \leq d_\xi$, and $1 \leq r,s \leq d_\eta.$
	\end{enumerate}
\end{prop}

\section{Constant coefficient vector fields}

Let $G_1$ and $G_2$ be compact Lie groups, $G:=G_1\times G_2$, and consider the linear operator $L:C^\infty(G)\to C^\infty(G)$  defined by
\begin{equation*}
L:=X_1+a X_2,
\end{equation*}
where $X_1 \in \mathfrak{g}_1$, $X_2 \in \mathfrak{g}_2$ and $a \in \C$. Thus, for each $u\in C^\infty(G)$ we have
\begin{align}
Lu(x_1,x_2) &:= X_1u(x_1,x_2) + a X_2u(x_1,x_2) \nonumber \\
& :=  \frac{d}{dt} u(x_1\exp(tX_1), x_2) \bigg|_{t=0} + a \frac{d}{ds} u(x_1, x_2\exp(sX_2))\bigg|_{s=0}. \nonumber
\end{align}

The operator $L$ extends to distributions in a natural way, that is, if $u \in \DG$, then
$$
\jp{Lu,\varphi}_G := -\jp{u,L\varphi}_G, \quad  \varphi \in C^\infty(G).
$$

In this section we present necessary and sufficient conditions for the vector field $L$ to be globally hypoelliptic and to be globally solvable. We also present examples recovering known results in the torus and presenting an example in $\TS$.

\subsection{Global hypoellipticity}

\begin{defi}
	Let $G$ be a compact Lie group. We say that an operator $P:\DG\to \DG$ is globally hypoelliptic if the conditions $u \in \DG$ and $Pu \in C^\infty(G)$ imply that $u \in C^\infty(G)$.
\end{defi}
\raggedbottom
Consider the equation
\begin{equation*}
Lu(x_1,x_2)=X_1u(x_1,x_2)+a X_2u(x_1,x_2)=f(x_1,x_2),
\end{equation*}
where $f \in C^{\infty}(G)$. For each $[\xi] \in \widehat{G_1}$, we can choose a representative $\xi \in \mbox{Rep}({G_1})$ such that
\begin{equation*}
\sigma_{X_1}(\xi)_{mn} =i\lambda_m(\xi) \delta_{mn}, \quad 1 \leq m,n \leq d_\xi,
\end{equation*}
where  $\lambda_m(\xi)\in\R$  for all $[\xi] \in\widehat{G_1}$ and $1 \leq m \leq d_\xi$.
Similarly, for each $[\eta] \in \widehat{G_2}$, we can choose a representative $\eta \in \mbox{Rep}({G_2})$ such that
\begin{equation*}
\sigma_{X_2}(\eta)_{rs} =i\mu_r(\eta) \delta_{rs}, \quad 1 \leq r,s \leq d_\eta,
\end{equation*}
where  $\mu_r(\eta)\in\R$  for all $[\eta] \in\widehat{G_2}$ and $1 \leq r \leq d_\eta$.

Suppose that $u \in C^\infty(G)$. Thus, taking the partial Fourier coefficient with respect to the first variable at $x_2 \in G_2$ we obtain
\begin{align*}
\widehat{f}(\xi,x_2)& = \widehat{Lu}(\xi,x_2) \\ &= \int_{G_1}Lu(x_1,x_2)\xi(x_1)^*\,dx_1 \nonumber \\ 
&= \int_{G_1} X_1u(x_1,x_2)\xi(x_1)^* \, dx_1 + a \int_{G_1}X_2u(x_1,x_2)\xi(x_1)^*\,dx_1 \nonumber \\ 
&= \widehat{X_1u}(\xi,x_2)+a X_2\int_{G_1}u(x_1,x_2)\xi(x_1)^{*} \, dx_1 \nonumber \\ 
&= \sigma_{X_1}(\xi)\widehat{u}(\xi,x_2)+a X_2 \widehat{u}(\xi,x_2). \nonumber 
\end{align*}
Hence, for each $x_2 \in G_2$, $\widehat{f}(\xi,x_2) \in \C^{d_\xi\times d_\xi}$  and 
\begin{equation*}
\widehat{f}(\xi,x_2)_{mn} = i\lambda_m(\xi)\widehat{u}(\xi,x_2)_{mn}+a X_2 \widehat{u}(\xi,x_2)_{mn}, \quad 1\leq m,n \leq d_\xi.
\end{equation*}

The details about partial Fourier series can be found in \cite{KMR19}. Now, taking the Fourier coefficient of $\widehat{f}(\xi, \cdot)_{mn}$ with respect to the second variable, we obtain
\begin{align}
\doublehat{\,f\,}\!(\xi,\eta)_{mn}&= \int_{G_2}\widehat{f}(\xi,x_2)_{mn}\eta(x_2)^*\,dx_2 \nonumber \\
&= \int_{G_2} (i\lambda_m(\xi)\widehat{u}(\xi,x_2)_{mn}+a X_2 \widehat{u}(\xi,x_2)_{mn})\eta(x_2)^* \,dx_2 \nonumber \\ 
&= i\lambda_m(\xi)\int_{G_2}\widehat{u}(\xi,x_2)_{mn}\eta(x_2)^* \,dx_2 + a\int_{G_2} X_2\widehat{u}(\xi,x_2)_{mn}\eta(x_2)^* \,dx_2 \nonumber \\ 
&= i\lambda_m(\xi)\doublehat{\, u \,}\!(\xi,\eta)_{mn}+ a \sigma_{X_2}(\eta)\doublehat{\, u \,}\!(\xi,\eta)_{mn}. \nonumber
\end{align}
Thus, $\doublehat{\,f\,}\!(\xi,\eta)_{mn} \in \C^{d_\eta\times d_\eta}$ and 
\begin{equation*}
\doublehat{\,f\,}\!(\xi,\eta)_{mn_{rs}}=i(\lambda_m(\xi)+a \mu_r(\eta))\doublehat{\, u\,}\!(\xi,\eta)_{mn_{rs}},\quad 1\leq r,s \leq d_\eta.
\end{equation*}

From this we can conclude that
\begin{equation}\label{image0}
\doublehat{\,f\,}\!(\xi,\eta)_{mn_{rs}}=0, \mbox{ whenever } \lambda_m(\xi)+a \mu_r(\eta) = 0.
\end{equation}
Moreover, if $\lambda_m(\xi)+a \mu_r(\eta) \neq 0$, then
\begin{equation}\label{fourieru}
\doublehat{\, u \,}\!(\xi,\eta)_{mn_{rs}} = \dfrac{1}{i(\lambda_m(\xi)+a\mu_{r}(\eta))} \doublehat{\,f\,}\!(\xi,\eta)_{mn_{rs}}.
\end{equation}

%

We begin by presenting the following necessary condition for global hypoellipticity of the vector field $L=X_1+a X_2$.

\begin{prop}\label{lemmainf}
	Suppose that the set
	\begin{equation}\label{setN}
		\mathcal{N}=\{([\xi], [\eta]) \in  \widehat{G_1} \times \widehat{G_2}; \ \lambda_m(\xi)+a \mu_r(\eta) = 0, \ \mbox{for some } 1 \leq m \leq d_\xi, 1 \leq r \leq d_\eta  \}
	\end{equation}

	has infinitely many elements. Then there exists $u \in \DG\setminus C^\infty(G)$ such that 
	$$
	Lu=0.
	$$
	In particular, $L$ is not globally hypoelliptic.
\end{prop}

\begin{proof}
Consider the sequence
$$\doublehat{\, u \,}\!(\xi,\eta)_{mn_{rs}} = \left\{
\begin{array}{ll}
1, & \mbox{if } \lambda_m(\xi)+a \mu_r(\eta)=0, \\
0, & \mbox{otherwise }
\end{array}
\right.
$$
Notice that for any $[\xi] \in \widehat{G_1}$, $[\eta] \in \widehat{G_2}$, $1\leq m,n \leq d_\xi$ and $1 \leq r,s \leq d_\eta$ we have
\begin{equation*}
|\doublehat{\, u \,}\!(\xi,\eta)_{mn_{rs}}| \nonumber\leq  \jp{\xi}+\jp{\eta}. \nonumber
\end{equation*}
Thus by the characterization of distributions by Fourier coefficients (Theorem \ref{smoopartial}) we conclude that $u\in\DG$, where
$$
u = \sum_{[\xi] \in \widehat{G_1}} \sum_{[\eta] \in \widehat{G_2}} d_\xi d_\eta \sum_{m,n=1}^{d_\xi} \sum_{r,s=1}^{d_\eta} \doublehat{\, u \,}\!(\xi,\eta)_{mn_{rs}} \xi_{nm} \eta_{sr}.
$$
Since there exist infinitely many representations such that $\doublehat{\, u \,}\!(\xi,\eta)_{mn_{rs}}=1$, it follows that $u \notin C^{\infty}(G)$. 
Furthermore, we have
$$
\doublehat{Lu}(\xi,\eta)_{mn_{rs}}=i(\lambda_m(\xi)+a \mu_r(\eta))\doublehat{\, u \,}\!(\xi,\eta)_{mn_{rs}} = 0, 
$$
for all  $[\xi] \in \widehat{G_1}, \ [\eta] \in \widehat{G_2},  1 \leq m \leq d_\xi,  1 \leq r \leq d_\eta.$ Then, by Plancherel formula \eqref{plancherel}, we conclude that  $Lu=0$. 
\end{proof}

\begin{thm}\label{thm1}
	The operator $L=X_1+a X_2$ is globally hypoelliptic  if and only if the following conditions are satisfied: 
	\begin{enumerate}[1.]
		\item The set
		$$
		\mathcal{N}=\{([\xi], [\eta]) \in  \widehat{G_1} \times \widehat{G_2}; \ \lambda_m(\xi)+a \mu_r(\eta) = 0, \ \mbox{for some } 1 \leq m \leq d_\xi, 1 \leq r \leq d_\eta  \}
		$$
		is finite.
		\item $\exists C, \, M>0$ such that
		\begin{equation}\label{hypothesis}
		|\lambda_m(\xi)+a\mu_r(\eta)|\geq C (\langle \xi\rangle +\langle \eta \rangle )^{-M}, 
		\end{equation}
		for all  $[\xi] \in \widehat{G_1}, \ [\eta] \in \widehat{G_2}, \ 1 \leq m \leq d_\xi, \  1 \leq r \leq d_\eta,$ whenever $\lambda_m(\xi)+a\mu_r(\eta) \neq 0$.
	\end{enumerate}
\end{thm}
\begin{proof}
$(\impliedby)$ Suppose that $Lu=f \in C^{\infty}(G)$ for some $u\in \DG$. Let us prove that $u \in C^{\infty}(G)$. 
Since the set $\mathcal{N}$ is finite, there exists $C>0$ such that
\begin{equation*}
|\doublehat{\, u \,}\!(\xi,\eta)_{mn_{rs}}| \leq C,
\end{equation*}
\raggedbottom
for all $([\xi], [\eta]) \in \mathcal{N}, \  1 \leq m,n \leq d_\xi, \ 1 \leq r,s \leq d_\eta$.
Let $N \in \N$. Then, for $([\xi], [\eta]) \in \mathcal{N}$, we have
\begin{align}
|\doublehat{\, u \,}\!(\xi,\eta)_{mn_{rs}}| &\leq C (\jp{\xi}+\jp{\eta})^N(\jp{\xi}+\jp{\eta})^{-N} \nonumber \\
&\leq C'_N(\jp{\xi}+\jp{\eta})^{-N} \nonumber
\end{align}
where ${C}'_N = \displaystyle\max_{([\xi],[\eta])\in \mathcal{N}}\left\{C (\langle \xi \rangle + \langle \eta \rangle)^{N}\right\}$.
On the other hand, if $([\xi],[\eta]) \notin \mathcal{N}$, by \eqref{fourieru} and \eqref{hypothesis} we obtain
\begin{align}
|\doublehat{\, u \,}\!(\xi,\eta)_{mn_{rs}}|
&\stackrel{}{=}\dfrac{1}{|\lambda_m(\xi)+a\mu_{r}(\eta)|} |\doublehat{\,f\,}\!(\xi,\eta)_{mn_{rs}}|\nonumber \\ 
&\stackrel{}{\leq}  C^{-1} (\langle \xi \rangle +\langle\eta \rangle) ^{M}|\doublehat{\,f\,}\!(\xi,\eta)_{mn_{rs}}|\nonumber
\end{align}
Since $f \in C^{\infty}(G)$, there exists $C_{N+M}>0$ such that
\begin{equation*}
|\doublehat{\,f\,}\!(\xi,\eta)_{mn_{rs}}|\leq C_{N+M}(\langle \xi \rangle+ \langle \eta \rangle)^{-(N+M)}
\end{equation*} 
Thus,
\begin{align}
|\doublehat{\, u \,}\!(\xi,\eta)_{mn_{rs}}|&\leq C^{-1}C_{N+M} (\langle \xi \rangle +\langle\eta \rangle) ^{M}(\langle \xi \rangle + \langle \eta \rangle)^{-(N+M)}\nonumber \\ 
&= {C''}_{\!\!\!N} (\langle \xi \rangle + \langle \eta \rangle)^{-N},\nonumber
\end{align}
where ${C''}_{\!\!\!N} =C^{-1}C_{N+M}$. Hence, if $([\xi],[\eta]) \notin \mathcal{N}$ we conclude that
\begin{equation*}
|\doublehat{\, u \,}\!(\xi,\eta)_{mn_{rs}}|\leq {C''}_{\!\!\!N}  (\langle \xi \rangle +\langle \eta \rangle)^{-N}.
\end{equation*}
Setting $C_N := \max\{{C'}_N, {C''}_{\!\!\!N} \}$, we have
$$
|\doublehat{\, u \,}\!(\xi,\eta)_{mn_{rs}}| \leq C_N (\langle \xi \rangle +\langle \eta \rangle)^{-N},
$$
for all $[\xi] \in \widehat{G_1}$, $[\eta] \in \widehat{G_2}$. Therefore by Theorem \ref{smoopartial} we conclude that $u \in C^{\infty}(G)$.

$(\implies)$ Let us prove the result by contradiction. If the condition 1 were not satisfied, by Proposition \ref{lemmainf}, there would be $u \in \DG \backslash C^\infty(G)$ such that $Lu=0$, contradicting the hypothesis of global hypoellipticity of $L$. So, let us assume that Condition 2 is not satisfied, then for every $M \in \N$, we choose $[\xi_M] \in  \widehat{G_1}$ and $[\eta_M] \in \widehat{G_2}$ such that
\begin{equation}\label{condition}
0<|\lambda_m(\xi_M)+a\mu_r(\eta_M)|\leq (\langle \xi_M \rangle +\langle \eta_M \rangle) ^{-M},
\end{equation}
for some $1 \leq m \leq d_{\xi_M}$ and $1 \leq r \leq d_{\eta_M}.$

Let $\mathcal{A}=\{([\xi_j],[\eta_j])\}_{j \in \N}$. It is easy to see that $\mathcal{A}$ has infinitely many elements. Define 
$$\doublehat{\, u \,}\!(\xi,\eta)_{mn_{rs}} = \left\{
\begin{array}{ll}
1, & \mbox{if $([\xi],[\eta])=([\xi_j],[\eta_j])$ for some $j\in \N$ and \eqref{condition} is satisfied} , \\
0, & \mbox{otherwise.}
\end{array}
\right.
$$
In this way,  $u \in \DG \backslash C^{\infty}(G)$. Let us show that $Lu=f \in C^{\infty}(G)$.

If $([\xi],[\eta]) \notin \mathcal{A}$, then $|\doublehat{\,f\,}\!(\xi,\eta)_{mn_{rs}}|=0$. Moreover, for every $M \in \N$, we have
\begin{align}
|\doublehat{\,f\,}\!(\xi_M,\eta_M)_{mn_{rs}}| &=  |\lambda_m(\xi_M)+a \mu_r(\eta_M)||\doublehat{\, u \,}\!(\xi_M,\eta_M)_{mn_{rs}}|\nonumber \\
&\leq(\langle \xi_M \rangle +\langle \eta_M\rangle) ^{-M}\nonumber
\end{align}
for every element of $\mathcal{A}$.

Fix $N>0 $. If $M> N$, then
$$
|\doublehat{\,f\,}\!(\xi_M,\eta_M)_{mn_{rs}}| \leq (\langle \xi_M \rangle + \langle \eta_M \rangle)^{-M} \leq  (\langle \xi_M \rangle + \langle \eta_M \rangle)^{-N} .
$$
For $M\leq N$ we have
\begin{align}
\left|\doublehat{\,f\,}\!(\xi_M,\eta_M)_{mn_{rs}}\right| &= \left|\doublehat{\,f\,}\!(\xi_M,\eta_M)_{mn_{rs}}\right|(\langle \xi_M \rangle + \langle \eta_M \rangle)^{N}(\langle \xi_M \rangle + \langle \eta_M \rangle)^{-N}\nonumber \\ 
& \leq  C'_N(\langle \xi_M \rangle + \langle \eta_M \rangle)^{-N}. \nonumber
\end{align}
where $C'_N:=\! \max\limits_{{M\leq N}}\left\{ |\doublehat{\,f\,}\!(\xi_M,\eta_M)_{mn_{rs}}|(\langle \xi_M \rangle + \langle \eta_M \rangle)^{N}\right\}$. 
For $C_N=\max\{C'_N,1\}$ we obtain
$$
|\doublehat{\,f\,}\!(\xi,\eta)_{mn_{rs}}| \leq C_N(\langle \xi \rangle + \langle \eta \rangle)^{-N},
$$
for all  $[\xi] \in \widehat{G_1}, \ [\eta] \in \widehat{G_2}, \ 1\leq m,n \leq d_\xi, \ 1\leq r,s \leq d_\eta$. Therefore $f \in C^{\infty}(G)$,  which contradicts the assumption that $L$ is globally hypoelliptic.
\end{proof}

\raggedbottom

\subsection{Global solvability}

In the literature there are several notions of the solvability of an operator, mainly depending on the functional environment in which one is working and what one intends to study. So the first step here is to define precisely what we mean by the global solvability.

Given a function (or distribution) $f$ defined on $G$, assume that $u\in\DG$ is a solution of $Lu=f$. By taking the partial Fourier coefficient with respect to $x_1$ and $x_2$ separately, and following the same procedure of the last subsection, we obtain from \eqref{image0} that
$$
\lambda_m(\xi)+a \mu_r(\eta)=0 \Longrightarrow \doublehat{\,f\,}\!(\xi,\eta)_{mn_{rs}}=0.
$$

Therefore, let us consider the following set 
$$\mathcal{K}:=\{w \in \DG; \ \doublehat{{\, w\,  }}(\xi,\eta)_{mn_{rs}}=0, \textrm{ whenever } \lambda_m(\xi)+a\mu_r(\eta)=0 \}.$$
If $f \notin \mathcal{K}$, then there is no $u \in \DG$ such that $Lu=f$.
We call the elements of $\mathcal{K}$ of admissible functions (distributions) for the solvability of $L$.

\begin{defi}
	We say that the operator $L$ is globally solvable if $L(\DG)=\mathcal{K}$.
\end{defi}

\begin{thm}\label{solvability}
	The operator $L=X_1+a X_2$ is globally solvable if and only if there exist $C,\,M>0$ such that
	\begin{equation}
	|\lambda_m(\xi)+a\mu_r(\eta)|\geq C (\langle \xi\rangle + \langle \eta \rangle )^{-M}, \label{condition1}
	\end{equation}
	for all  $[\xi] \in \widehat{G_1}, \ [\eta] \in \widehat{G_2},\  1 \leq m \leq d_\xi,\  1 \leq r \leq d_\eta$ whenever $\lambda_m(\xi)+a\mu_r(\eta) \neq 0$.
\end{thm}
\begin{proof}
$(\impliedby)$ For each $f \in \mathcal{K}$ define
\begin{equation}\label{solution1}
\doublehat{\, u \,}\!(\xi,\eta)_{mn_{rs}} = \left\{
\begin{array}{ll}
0, & \mbox{if } \lambda_m(\xi)+ a \mu_r(\eta)=0, \\
-i({\lambda_m(\xi)+ a \mu_r(\eta)})^{-1} \doublehat{\,f\,}\!(\xi,\eta)_{mn_{rs}}, & \mbox{otherwise. }
\end{array}
\right.
\end{equation}
Let us show that $\{ \doublehat{\, u \,}\!(\xi,\eta)_{mn_{rs}} \}$ is the sequence of Fourier coefficient of an element $u \in \DG$. Since $f \in \DG$, there exists $N \in \N$ and $C>0$ such that
$$
|\doublehat{\,f\,}\!(\xi,\eta)_{mn_{rs}}| \leq C(\langle \xi \rangle + \langle \eta \rangle)^{N},
$$
for all  $[\xi] \in \widehat{G_1}, \ [\eta] \in \widehat{G_2},\  1 \leq m \leq d_\xi,\  1 \leq r \leq d_\eta$. So
\begin{align}
|\doublehat{\, u \,}\!(\xi,\eta)_{mn_{rs}}|&= |{\lambda_m(\xi)+ a \mu_r(\eta)}|^{-1} |\doublehat{\,f\,}\!(\xi,\eta)_{mn_{rs}}|\label{construct}\\
&\leq C(\langle \xi \rangle +\langle \eta \rangle)^{M} |\doublehat{\,f\,}\!(\xi,\eta)_{mn_{rs}}| \nonumber \\
&\leq C(\jp{\xi} + \jp{\eta})^{N+M} \nonumber 
\end{align}
Therefore $u \in \DG$ and $Lu=f$.

$(\implies)$ Let us proceed by contradiction by constructing an element $f \in \mathcal{K}$ such that there is no $u \in \DG$ satisfying $Lu=f$.

If \eqref{condition1} is not satisfied, for each $M \in \N$, there exists $[\xi_M] \in \widehat{G_1}$ and $[\eta_M] \in \widehat{G_2}$ such that
\begin{equation}\label{condition3}
0 < |\lambda_{\tilde{m}}(\xi_M)+ a \mu_{\tilde{r}}(\eta_M)| < (\langle \xi_M \rangle + \langle \eta_M \rangle)^{-M},
\end{equation}
for some $1 \leq \tilde{m} \leq d_{\xi_M}$ and $1 \leq \tilde{r} \leq d_{\eta_M}$. We can suppose that $\jp{\xi_M}+\jp{\eta_M} \leq \jp{\xi_N}+\jp{\eta_N}$ when $M \leq N$.
Let $\mathcal{A}=\{([\xi_j],[\eta_j])\}_{j \in \N}$. Consider $f\in \mathcal{K}$ defined by
$$
\doublehat{\,f\,}\!(\xi,\eta)_{mn_{rs}} = \left\{
\begin{array}{ll}
1, &  \mbox{if $([\xi],[\eta])=([\xi_j],[\eta_j])$ for some $j\in \N$ and \eqref{condition3} is satisfied}, \\
0, & \mbox{otherwise. }
\end{array}
\right.
$$

Suppose that there exits $u \in \DG$ such that $Lu=f$. In this way, its Fourier coefficients must satisfy
$$
i(\lambda_m(\xi)+ a \mu_r(\eta))\doublehat{\, u \,}\!(\xi,\eta)_{mn_{rs}}=\doublehat{\,f\,}\!(\xi,\eta)_{mn_{rs}}.
$$
So
\begin{align}
|\doublehat{\, u \,}\!(\xi_M,\eta_M)_{\tilde{m}1_{\tilde{r}1}}| &= |{\lambda_{\tilde{m}}(\xi_M)+ a \mu_{\tilde{r}}(\eta_M)}|^{-1}| |\doublehat{\,f\,}\!(\xi_M,\eta_M)_{\tilde{m}1_{\tilde{r}1}}| \nonumber\\
&>  (\langle \xi_M \rangle + \langle \eta_M \rangle)^{M},\nonumber
\end{align}
where $\tilde{m}$ and $\tilde{r}$ are coefficients that satisfy \eqref{condition3}. Thus
$$
\|\doublehat{\, u \,}\!(\xi_M,\eta_M)\|_\HS > (\langle \xi_M \rangle + \langle \eta_M \rangle)^{M},
$$
for all $M>0$, which contradicts the fact that $ u \in \DG$. Therefore there does not exist $u \in \DG$ such that $Lu=f$.
\end{proof}

Notice that the estimate for the global solvability in the statement of the last theorem is exactly the same as one of the conditions to obtain global hypoellipticity announced in \eqref{hypothesis}, thus we have the following corollary.

\begin{cor}
	If $L$ is globally hypoelliptic, then $L$ is globally solvable.
\end{cor}

A more detailed analysis of the last proof shows that it is possible to obtain a better control on the Fourier coefficients of $u$ when $f$ is smooth, more precisely, we have the following result.

\begin{prop}\label{solvabilitysmooth}
	If $L$ is globally solvable and $f \in \mathcal{K}\cap C^{\infty}(G)$, then there exists $u \in C^{\infty}(G)$ such that $Lu=f$. 
\end{prop}
\begin{proof}

Let $f \in \mathcal{K}\cap C^{\infty}(G)$ and define $u$ as in \eqref{solution1}. Since $L$ is globally solvable, we have \eqref{condition1}, and then by \eqref{construct}
\begin{align}
|\doublehat{\, u \,}\!(\xi,\eta)_{mn_{rs}}|&= |{\lambda_m(\xi)+ a \mu_r(\eta)}|^{-1} |\doublehat{\,f\,}\!(\xi,\eta)_{mn_{rs}}|\nonumber\\
&\leq C(\langle \xi \rangle +\langle \eta \rangle)^{M}|\doublehat{\,f\,}\!(\xi,\eta)_{mn_{rs}}|\nonumber.
\end{align}
In view of the smoothness of $f$, for every $N>0$ there exists $C_N>0$ such that
$$
|\doublehat{\,f\,}\!(\xi,\eta)_{mn_{rs}}|\leq C_N(\jp{\xi}+\jp{\eta})^{-N},
$$
for all   $[\xi] \in \widehat{G_1}, \ [\eta] \in \widehat{G_2},\  1 \leq m \leq d_\xi,\  1 \leq r \leq d_\eta$.
Hence
$$
|\doublehat{\, u \,}\!(\xi,\eta)_{mn_{rs}}| \leq C(\langle \xi \rangle +\langle \eta \rangle)^{M}|\doublehat{\,f\,}\!(\xi,\eta)_{mn_{rs}}| \leq C_{N+M}(\jp{\xi}+\jp{\eta})^{-N}.
$$
Therefore $u\in C^\infty(G)$ and $Lu=f$.
\end{proof}
\begin{rem}
	The proposition above says that we can obtain a smooth solution for $Lu=f$ in the case where $L$ is globally solvable and $f$ is a smooth admissible function. Notice that this does not mean that $L$ is globally hypoelliptic.
\end{rem}
\subsection{Examples}

In this section we recover some classical examples of S. Greenfield and N. Wallach (see \cite{GW72}), on the global hypoellipticity 
and global solvability in tori ($\mathbb{T}^2$ and $\mathbb{T}^d$) and present a class of examples in $\mathbb{T}^1\times\St$.

\begin{ex}{$G=\mathbb{T}^2$}\label{exampleliouville}
	\end{ex}
Set $G_1=G_2=\mathbb{T}^1$, where $\mathbb{T}^1= \R/2\pi\Z$. Since $\mathbb{T}^1$ is abelian, the irreducible unitary representations of $\mathbb{T}^1$ are unidimensional. Moreover the dual $\widehat{\mathbb{T}^1}$ can be identified to $\Z$. 
For each $k\in \Z$, the function $e_{k}: \mathbb{T}^1 \to \mathcal{U}(\C)$ defined by
$$
e_k(t):= e^{i tk}
$$
is an element of $\widehat{\mathbb{T}^1}$ and
$$
\widehat{\mathbb{T}^1} \cong \{e_{k}\}_{k \in \Z}.
$$
The Haar measure on $\mathbb{T}^1$ is the normalized Lebesgue measure and
$$
\jp{k} := \jp{e_k} = \sqrt{1+k^2}. 
$$

Let $a \in \C$ and consider the operator
$$
L=\partial_t + a \partial_x, \quad (t,x) \in \mathbb{T}^1 \times \mathbb{T}^1.
$$
Notice that
$$
\sigma_{\partial_t}(e_k) = e_k(t)^* (\partial_t e_k)(t) = e^{-itk} (ik e^{itk}) = ik,
$$
that is, $\lambda(e_k) = k$, for all $k \in \Z$. Thus, if $Lu=f$, then
$$
\doublehat{\,f\,}\!(k,\ell) = i(k+a \ell)\ \doublehat{\, u \,}\!(k,\ell).
$$
In this case,
$$
\mathcal{N} = \{(k,\ell) \in \Z^2; \ k+a \ell = 0  \}.
$$
By Theorem \ref{thm1}, $L$ is globally hypoelliptic if and only if $\mathcal{N}$ is finite and there exist $C,M>0$ such that
$$
|k+a \ell| \geq C (\jp{k} + \jp{\ell})^{-M}
$$
for all $(k,\ell) \in \Z^2$, whenever $k+a \ell \neq 0$.
For $(k,\ell) \neq (0,0)$, we have
$$
|k|+|\ell| \leq \jp{k}+ \jp{\ell} \leq 3(|k|+|\ell|),
$$
then the second condition of the Theorem \ref{thm1} becomes
\begin{equation}\label{liouville}
|k+a \ell| \geq C (|k|+|\ell|)^{-M}
\end{equation}
for all $(k,\ell) \in \Z^2$, whenever $k+a \ell \neq 0$. 

Notice that $\mathcal{N}$ is an infinity set if and only if $a \in \Q$. Moreover, if $a \notin \Q$, then $\mathcal{N} = \{ (0,0)\}$. Suppose that $\mbox{Im} (a) \neq 0$. If $\ell \neq 0$, then
$$
|k+a \ell| \geq |\mbox{Im}(a)| |\ell| \geq |\mbox{Im}(a)| (|k|+|\ell|)^{-1}.
$$
If $\ell=0$, we have $k\neq0$ and
$$
|k+a \ell| = |k| \geq (|k|+|\ell|)^{-1}.
$$
Take $C = \max \{1, |\mbox{Im}(a)|\}$. Then
$$
|k+a \ell| \geq C(|k|+|\ell|)^{-1},
$$
for all $(k,\ell) \in \Z^2 \setminus \{(0,0) \}$. Therefore, if $\mbox{Im}(a)\neq 0$ then $L$ is globally hypoelliptic.

Suppose now that $\mbox{Im}(a) = 0$. We recall that an irrational number $a$ is called a Liouville number if it can be approximated by rational numbers to any order. That is, for every positive integer $N$ there is $K>0$ and infinitely many integer pairs $(k,\ell)$ so that 
$$
\left|a-\frac{k}{\ell}\right| < \frac{K}{\ell^N}.
$$

 Notice that the inequality \eqref{liouville} is satisfied if and only if $a$ is an irrational non-Liouville number. 

We conclude that $L=\partial_t+a \partial_x$ is globally hypoelliptic if and only if $\mbox{Im}(a)\neq0$ or $a$ is an irrational non-Liouville number.

For solvability we need to analyze the condition 2 of the Theorem \ref{thm1} when $a \in \Q$. Suppose that $a = \frac{p}{q}$, $p\in \Z$ and $q \in \N$. We have
$$
|k+a \ell| = \left|k+ \frac{p}{q}\ell\right| = \frac{1}{q}|qk+p\ell| \geq \frac{1}{q} \geq \frac{1}{q}(|k|+|\ell|)^{-1},
$$
for all $(k,\ell)\in \Z^2$, whenever $qk+p\ell \neq 0$. 

Therefore, $L=\partial_t+a \partial_x$ is globally solvable if and only if $\mbox{Im}(a)\neq 0$, or $a \in \Q$, or $a$ is an irrational non-Liouville number.

\begin{ex}{$G=\mathbb{T}^d$}
	\end{ex}
From the above example we can extend the analysis for operators defined on $\mathbb{T}^d$. Let
$$
L=\sum_{j=1}^d a_j \partial_{t_j}, \quad a_j\in \C 
$$
If $Lu=f$, then
$$
\widehat{f}(k_1, \cdots, k_d) = i\left(\sum_{j=1}^d a_j k_j\right) \widehat{u}(k_1, \cdots, k_d).
$$
The set $\mathcal{N}$ is
$$
\mathcal{N} = \left\{k \in \Z^d ;  \sum_{j=1}^d a_jk_j=0 \right\},
$$
and by Theorem \ref{thm1}, $L$ is globally hypoelliptic if and only if $\mathcal{N}$ is finite and there exists $C,M>0$ such that
$$
\left| \sum_{j=1}^d a_j k_j\right| \geq C \left(\sum_{j=1}^d |k_j|\right)^{-M},
$$
for all $k\in \Z^d$ whenever $\sum\limits_{j=1}^d a_j k_j \neq0$. 

For instance, if some $a_j=0$, then the set $\mathcal{N}$ is infinity, which implies that $L$ is not globally hypoelliptic. It is easy to see that if all $a_j \in \Q$, then $L$ is globally solvable, even if some of $a_j=0$.

If $a_j=1$ for $j=1, \cdots, d-1$ and $\mbox{Im}(a_d)\neq 0$, than $L$ is globally hypoelliptic. The same is true if we consider $a_d$ being an irrational non-Liouville number.

\begin{ex}{$G=\mathbb{T}^1 \times \mathbb{S}^3$}\label{examples3}
	\end{ex}
Let $\Sth$ be the unitary dual of $\St$, that is, $\Sth$ consists of equivalence classes $[\textsf{t}^\ell]$ of continuous irreducible unitary representations
$\textsf{t}^\ell:\St\to \C^{(2\ell+1)\times (2\ell+1)}$, $\ell\in\frac12\N_{0}$,
of matrix-valued functions satisfying
$\textsf{t}^\ell(xy)=\textsf{t}^\ell(x)\textsf{t}^\ell(y)$ and $\textsf{t}^\ell(x)^{*}=\textsf{t}^\ell(x)^{-1}$ for all
$x,y\in\St$.
We will use the standard convention of enumerating the matrix elements
$\textsf{t}^\ell_{mn}$ of $\textsf{t}^\ell$ using indices $m,n$ ranging between
$-\ell$ to $\ell$ with step one, i.e. we have $-\ell\leq m,n\leq\ell$
with $\ell-m, \ell-n\in\N_{0}.$ For $\ell \in \tfrac{1}{2} \N_0$ we have
$$
\jp{\ell} := \jp{\textsf{t}^\ell} = \sqrt{1+\ell(\ell+1)}. 
$$The details about the Fourier analysis on $\St$ can be found in Chapter 11 of \cite{livropseudo}.

Let $X$ be a smooth vector field on $\St$ and $a \in \C$. Consider the following operator defined on $\mathbb{T}^1 \times \St$:
$$
L=\partial_t + a X.
$$
Using rotation on $\St$, without loss of generality, we may assume that the vector
field $X$ has the symbol 
$$
\sigma_{X}(\ell)_{mn}=im\delta_{mn}, \quad \ell \in \tfrac{1}{2}\N_0, \ -\ell\leq m,n\leq \ell, \ \ell-m, \ell-n \in \N_0,
$$ 
with $\delta_{mn}$ standing for the Kronecker's delta (see \cite{livropseudo}, \cite{RT13}, and \cite{RTW14}). Hence, if $Lu=f$, then
$$
\doublehat{\,f\,}\!(k,\ell)_{mn} = i(k+a m) \doublehat{\, u \,}\!(k,\ell)_{mn},
$$
where $k \in \Z$, $\ell \in \frac{1}{2}\N_0$, $-\ell \leq m,n \leq \ell$ and $\ell-m, \ell-n \in \N_0$. In this case,
$$
\mathcal{N} = \{(k,\ell) \in \Z \times \tfrac{1}{2}\N_0;\  k+a m = 0, \mbox{ for some } -\ell \leq m \leq \ell, \ell-m \in \N_0 \}.
$$
By Theorem \ref{thm1}, $L$ is globally hypoelliptic if and only if $\mathcal{N}$ is finite and there exist $C,M>0$ such that
\begin{equation}\label{condi}
|k+a m| \geq C (\jp{k} + \jp{\ell})^{-M}
\end{equation}
for all $(k,\ell) \in \Z\times \tfrac{1}{2}\N_0$, $ -\ell \leq m \leq \ell, \ \ell-m \in \N_0$ whenever $k+a m \neq 0$.
For $\ell \in \tfrac{1}{2}\N_0$, we have
$$
\frac{1}{\sqrt{2}}(1+\ell) \leq \langle \textsf{t}^\ell \rangle \leq 1+\ell
$$
and we can write \eqref{condi} as
\begin{equation*}
|k+a m| \geq C (|k|+1+\ell)^{-M}
\end{equation*}
for all $(k,\ell) \in \Z\times \tfrac{1}{2}\N_0$, $ -\ell \leq m \leq \ell, \ \ell-m \in \N_0$ whenever $k+a m \neq 0$.

Notice that $(0,\ell) \in \mathcal{N}$, for all $\ell \in \N_0$, so $\mathcal{N}$ has infinitely many elements and then $L$ is not globally hypoelliptic for any $a \in \C$. 

The analysis of the global solvability of $L$ is similar to the $\mathbb{T}^2$ case and we have that  $L$ globally solvable if and only if $\mbox{Im}(a) \neq0$, or $a \in \Q$, or $a$ is an irrational non-Liouville number.

\section{Weaker notions of hypoellipticity}\label{weak}
All the known examples of globally hypoelliptic vector fields are set on tori. Actually, in 1973, S. Greenfield and N. Wallach proposed the following conjecture.

\begin{conjecture}[Greenfield-Wallach]\label{GW} If a closed, connected, orientable manifold $M$ admits a globally hypoelliptic vector field $X$, then $M$ is diffeomorphic to a torus and $X$ is smoothly conjugate to a constant Diophantine vector field.
\end{conjecture}

In \cite{F08}, G. Forni showed the equivalence between this conjecture and Katok's conjecture, about the existence of $C^\infty$--co\-ho\-mo\-lo\-gy free smooth vector fields on closed, connected, orientable smooth manifolds. From this equivalence we will show that on compact connected Lie groups the set $\mathcal{N}$ defined in \eqref{setN} contains only the trivial representation. First, let us define what is a $C^\infty$--cohomology free vector field.
\begin{defi}\label{CH}
	Let $M$ be a closed, connected, orientable smooth manifold. A smooth vector field $X$ on $M$ is $C^\infty$--cohomology free if for all $f \in C^\infty(M)$ there exists a constant $c(f) \in \C$ and $u \in C^\infty(M)$ such that
	$$
	Xu=f-c(f).
	$$
\end{defi}

\begin{thm}\label{fornithm}[G. Forni \cite{F08}]
	Let $X$ be a smooth vector field on a closed connected manifold $M$. Then $X$ is $C^\infty$--cohomology free if and only if $X$ is globally hypoelliptic.
\end{thm}

\begin{prop}\label{GHtrivial}
	If $G$ is a compact connected Lie group and $L$ is globally hypoelliptic, then $\mathcal{N} $ has only one element.
\end{prop}
\begin{proof}
	Notice that for the trivial representations $\mathds{1}_{G_1}$ and $\mathds{1}_{G_2}$ we have $\lambda_1(\mathds{1}_{G_1})=\mu_1(\mathds{1}_{G_2})=0$, so $\mathcal{N} \neq \varnothing$.
	Suppose that there exists a non-trivial representation such that
	$$
	\lambda_m(\xi)+a \mu_r(\eta) = 0.
	$$
	for some $1 \leq m \leq d_\xi$, $1 \leq r \leq d_\eta$.
	Let $f = \xi_{{1m}} \times \eta_{{1r}} \in C^\infty(G)$, so
	
	\begin{align}
	\doublehat{\,f\,}\!(\xi,\eta)_{m1_{r1}}&= \int_{G_1}\int_{G_2} f(x_1,x_2) \overline{\xi(x_1)_{1m}} \, \overline{\eta(x_2)_{1r}} dx_2 dx_1\nonumber \\
	&= \int_{G_1}\int_{G_2} \xi(x_1)_{1m}\eta(x_2)_{1r} \overline{\xi(x_1)_{1m}} \, \overline{\eta(x_2)_{1r}} dx_2 dx_1 \nonumber \\
	&= \int_{G_1} |\xi(x_1)_{1m}|^2 \, dx_1 \int_{G_2} |\eta(x_2)_{1r}|^2 \, dx_2 \nonumber \\
	&= (d_\xi d_\eta)^{-1} \nonumber .
	\end{align}
	Since $L$ is globally hypoelliptic, by Theorem \ref{fornithm} $L$ is $C^\infty$--cohomology free, then there exists $u \in C^\infty(G)$  such that
	$$
	Lu=f-f_0,
	$$
	where $f_0=\int_{G} f \, d\mu_G$.
	We have
	$$
	\doublehat{Lu}(\xi,\eta)_{m1_{r1}} = i(\lambda_m(\xi))+a\mu_r(\eta))\doublehat{\, u \,}\!(\xi,\eta)_{m1_{r1}} = 0,
	$$
	which implies that
	$$
	\doublehat{{f-f_0}}(\xi,\eta)_{m1_{r1}}= 0. 
	$$
	Since $\xi\otimes\eta$ is not the trivial representation, by \eqref{ortho} we have $\doublehat{{f_0}}(\xi,\eta)_{m1_{r1}}=0$, so 
	$$
	\doublehat{\,f\,}\!(\xi,\eta)_{m1_{r1}}= 0, 
	$$
	what is a contradiction because $\doublehat{\,f\,}\!(\xi,\eta)_{m1_{r1}}=(d_\xi d_\eta)^{-1} $. Therefore $\mathcal{N}$ contains only the trivial representation.
\end{proof}

In view of Example \ref{examples3} and Proposition \ref{GHtrivial}, the following question naturally arises:

\begin{question}\label{question}
	Does there exist a compact Lie group $G \neq \mathbb{T}^d$ such that there exists $X\in \mathfrak{g}$ such that $\sigma_X(\phi)$ is singular for only finitely many $[\phi] \in \widehat{G}$, that is, the set
	$$
	\mathcal{Z} = \{[\phi] \in \widehat{G}; \ \lambda_m(\phi)=0, \mbox{ for some }  1 \leq m \leq d_\phi\}
	$$
	is finite, where $\sigma_X(\phi)_{mn}=i\lambda_m(\phi)\delta_{mn}$?
\end{question}

The Greenfield-Wallach conjecture was only proved in dimensions 2 and 3, and in some very particular cases, which are described by G. Forni in \cite{F08}. We suspect the answer to the above question is a way to prove the conjecture for Lie groups.

In view of the probable validity of the Greenfield-Wallach conjecture, the study of the global hypoellipticity of vector fields defined on closed manifolds is restricted to tori. However, the study of the regularity of solutions of such vector fields is yet an interesting subject. For this reason, in this section we will make some considerations looking to weaken the usual concept of the  global hypoellipticity and introduce what we will call global hypoellipticity modulo kernel and global $\mathcal{W}$-hypoellipticity.

\subsection{Global hypoellipticity modulo kernel}
\mbox{ }

First, assuming that the set $\mathcal{N}$ has infinitely many elements, we will show that to reduce the range of the operator does not help us to obtain a weaker version of global hypoellipticity.

\begin{prop}\label{image}
	Suppose that $\mathcal{N}$ has infinitely many elements. Then there is no subset $\mathcal{A} \subseteq C^{\infty}(G)$ that satisfies the condition: $u \in \DG$ and $Lu \in \mathcal{A}$ imply that $u \in C^\infty(G)$.
\end{prop}
\begin{proof}
	Assume that there exists a subset $\mathcal{A} \subseteq C^{\infty}(G)$ that satisfies the property above. Let $u \in \DG$ such that $Lu \in \mathcal{A}$, then $u \in C^{\infty}(G)$. By Proposition \ref{lemmainf} there exists an element $v \in \ker L$ such that $v \in \DG \backslash C^\infty(G)$.  Since $v \in \ker L$, we have $L(u+v) = Lu \in \mathcal{A}$, which implies that $u+v \in C^{\infty}(G)$. Therefore $v = (u+v)-u \in C^{\infty}(G)$, a contradiction.
\end{proof}
In view of Proposition \ref{image} we give the following definition:

\begin{defi}
	We say that an operator $P:\DG \to \DG$ is globally hypoelliptic modulo $\ker P$ if the conditions $u \in \DG$ and $Pu \in C^\infty(G)$ imply that there exists $v \in C^\infty(G)$ such that $u-v \in \ker P$.
\end{defi}

Clearly, global hypoellipticity implies global hypoellipticity modulo kernel. Our main result here is the equivalence of the concepts of global hypoellipticity modulo kernel and global solvability for constant coefficient vector fields.

\begin{prop}\label{modker}
	The operator $L=X_1+a X_2$ is globally hypoelliptic modulo $\ker L$ if and only if $L$ is globally solvable.
\end{prop}
\begin{proof}
	$(\implies)$ Suppose that $L$ is not globally solvable. Then by Theorem \ref{solvability}, for every $M \in \N$, choose $[\xi_M] \in  \widehat{G_1}$ and $[\eta_M] \in \widehat{G_2}$ such that
	\begin{equation*}
	0<|\lambda_m(\xi_M)+a\mu_r(\eta_M)|\leq (\langle \xi_M \rangle +\langle \eta_M \rangle) ^{-M},
	\end{equation*}
	for some $1 \leq m \leq d_{\xi_M}$ and $1 \leq r \leq d_{\eta_M}.$ Using the same construction of the proof of Theorem \ref{thm1}, we find a $u \in \DG\setminus C^\infty(G)$ such that $Lu=f\in C^\infty(G)$. Notice that if $u-v \in \ker L$, for some $v \in C^\infty(G)$, then
	$$
	i(\lambda_m(\xi)+a \mu_r(\eta))\doublehat{u-v}(\xi,\eta)_{mn_{rs}} = 0,
	$$
	for all   $[\xi] \in \widehat{G_1}, \ [\eta] \in \widehat{G_2},\  1 \leq m,n \leq d_\xi,\  1 \leq r,s \leq d_\eta$, which implies that
	$$
	\lambda_m(\xi)+a \mu_r(\eta) \neq 0 \implies \doublehat{u}(\xi,\eta)_{mn_{rs}} = \doublehat{v}(\xi,\eta)_{mn_{rs}}.
	$$
	Since $\doublehat{u}(\xi_M,\eta_M)_{mn_{rs}}=1$, we conclude that $v \notin C^\infty(G)$, so $L$ is not globally hypoelliptic modulo $\ker L$.
	
	$(\impliedby)$ Let $u\in \DG$ such that $Lu=f \in C^\infty(G)$. Notice that $f \in \mathcal{K}\cap C^\infty(G)$ and by Proposition \ref{solvabilitysmooth} there exists $v \in C^\infty(G)$ such that $Lv=f$. Therefore $u-v \in \ker L$ and then $L$ is globally hypoelliptic modulo $\ker L$.
\end{proof}

\begin{ex}
	Let $G=\TS$. In Example \ref{examples3} we saw that the operator $L=\partial_t + X$ is not globally hypoelliptic but it is globally solvable. By Proposition \ref{modker}, we conclude that  even not being globally hypoelliptic, the operator $L$ is globally hypoelliptic modulo kernel.
\end{ex}
\subsection{$\mathcal{W}$--global hypoellipticity}

In the light of Proposition \ref{image}, our next notion of hypoellipticity is based on the reduction of the domain of the operator.

\begin{defi}
	Let $\mathcal{W}$ be a subset of $\DG$. We say that an operator $P:\DG \to \DG$ is $\mathcal{W}$-globally hypoelliptic if the conditions $u \in \mathcal{W}$ and $Pu \in C^\infty(G)$ imply that $u \in C^\infty(G)$.
\end{defi}

Observe that an operator $P$ is always $C^{\infty}(G)$--globally hypoelliptic, and to say that $P$ is $\DG$-globally hypoelliptic means that $P$ is globally hypoelliptic.

\begin{ex}\label{khypo}
	Let $L=X_1 + a X_2$ and set $$\mathcal{K}:=\{u \in \DG; \ \doublehat{\, u \,}\!(\xi,\eta)_{mn_{rs}}=0, \textrm{whenever } \lambda_m(\xi)+a\mu_r(\eta)=0 \}.$$ If $L$ is globally solvable, then $L$ is $\mathcal{K}$-globally hypoelliptic.
	
	Indeed, by the characterization of the global solvability (Theorem \ref{solvability}), there exist $C, \, M>0$ such that
	\begin{equation*}
	|\lambda_m(\xi)+a\mu_r(\eta)|\geq C (\langle \xi\rangle +\langle \eta \rangle )^{-M}, 
	\end{equation*}
	for all  $[\xi] \in \widehat{G_1}, \ [\eta] \in \widehat{G_2}, \ 1 \leq m \leq d_\xi, \  1 \leq r \leq d_\eta,$ whenever $\lambda_m(\xi)+a\mu_r(\eta) \neq 0$.
	
	Let $u \in \mathcal{K}$ such that $Lu=f \in C^\infty(G)$. We know that
	$$
	\doublehat{\,f\,}\!(\xi,\eta)_{mn_{rs}}=i(\lambda_m(\xi)+a \mu_r(\eta))\doublehat{\, u \,}\!(\xi,\eta)_{mn_{rs}}.
	$$
	
	If $\lambda_m(\xi)+a\mu_r(\eta)=0$ then $\doublehat{\, u \,}\!(\xi,\eta)_{mn_{rs}}=0$. 
	
	If $\lambda_m(\xi)+a\mu_r(\eta) \neq 0$, we have
	$$
	|\doublehat{\, u \,}\!(\xi,\eta)_{mn_{rs}}| = \frac{1}{|\lambda_m(\xi)+a\mu_r(\eta)|} |	\doublehat{\,f\,}\!(\xi,\eta)_{mn_{rs}}| \leq C (\langle \xi\rangle +\langle \eta \rangle )^{M}|\doublehat{\,f\,}\!(\xi,\eta)_{mn_{rs}} |.
	$$
	Therefore $u \in C^\infty(G)$.
\end{ex}

\begin{prop}\label{inclusion}
	If $\mathcal{W}_1 \subseteq \mathcal{W}_2$ and $L$ is $\mathcal{W}_2$--globally hypoelliptic, then $L$ is $\mathcal{W}_1$--globally hypo\-elliptic.
\end{prop}

\begin{proof} Let $u \in \mathcal{W}_1$ such that $Lu \in C^{\infty}(G)$. As $\mathcal{W}_1 \subseteq \mathcal{W}_2$, we have $u \in\mathcal{W}_2$ and since $L$ is $\mathcal{W}_2$--globally hypoelliptic, $u \in C^{\infty}(G)$. Therefore $L$ is $\mathcal{W}_1$--globally hypoelliptic. 
\end{proof}

\begin{cor}
	If $L$ is globally solvable, then $L$ is $L(\DG)$--globally hypoelliptic, where $L(\DG)$ denotes the image of $L$.
\end{cor}

\begin{proof}
If $f \in L(\DG)$, then there exists $u \in \DG$ such that $Lu=f$, which implies that $f \in \mathcal{K}$, so $L(\DG) \subset \mathcal{K}.$ Since $L$ is $\mathcal{K}$--globally hypoelliptic (Example \ref{khypo}), by Proposition \ref{inclusion} we conclude that $L$ is $L(\DG)-$hypoelliptic.
\end{proof}

\begin{cor}
	Suppose that $L$ is globally solvable. If there exists $k \in \N$ such that $L^ku \in C^{\infty}(G)$, then $Lu \in C^{\infty}(G)$.
\end{cor}

\begin{proof}
Suppose that there exists $k \in \N$ such that $L^ku \in C^{\infty}(G)$. Since $v=L^{k-1}u \in L(\DG)$ and $Lv \in C^{\infty}(G)$, we have, by the $L(\DG)$--global hypoellipticity of $L$, that $v \in L^{k-1}u \in C^{\infty}(G)$. We can continue this process to conclude that $Lu \in C^{\infty}(G)$.
\end{proof}

If $L$ is globally solvable, the previous corollary says that if $Lu \notin C^{\infty}(G)$, then $L^ku \notin C^{\infty}(G)$ for all $k \in \N$.

Let 
\begin{equation*}
\mathcal{M}:=\{u \in \DG; \forall N \!\in\! \N, \exists C_N>0;  \|\doublehat{\, u \,}\!(\xi,\eta)\|_{\HS} \leq C_N (\langle \xi \rangle + \langle \eta \rangle)^{-N}\!, ([\xi], [\eta])\! \in\! \mathcal{N} \}.
\end{equation*}
Notice that $C^{\infty}(G) \subsetneq \mathcal{M}$.
\begin{thm}
	If $L$ is globally solvable, then $L$ is $\mathcal{M}$--globally hypoelliptic. 
\end{thm}
\begin{proof}
Let $u \in \mathcal{M}$ such that $Lu \in C^{\infty}(G)$. We know that
$$
\doublehat{Lu}(\xi,\eta)_{mn_{rs}}=i(\lambda_m(\xi)+a \mu_r(\eta))\doublehat{\, u \,}\!(\xi,\eta)_{mn_{rs}},
$$
for all  $[\xi] \in \widehat{G_1}, \ [\eta] \in \widehat{G_2},  1 \leq m \leq d_\xi,  1 \leq r \leq d_\eta.$ If $([\xi],[\eta]) \notin \mathcal{N}$, then $\lambda_m(\xi)+a \mu_r(\eta)\neq0$ and
\begin{equation*}
\doublehat{\, u \,}\!(\xi,\eta)_{mn_{rs}} = \dfrac{1}{i(\lambda_m(\xi)+a\mu_{r}(\eta))} \doublehat{Lu}(\xi,\eta)_{mn_{rs}}.
\end{equation*}
Proceeding similarly as in the proof of Theorem \ref{thm1}, it can be proved that for every $N \in \N$, there exists $C'_N>0$ such that
$$\|\doublehat{\, u \,}\!(\xi,\eta)\|_{\HS} \leq C'_N(\langle \xi \rangle + \langle \eta \rangle)^{-N},
$$
for all $([\xi],[\eta]) \notin \mathcal{N}$. Since $u \in \mathcal{M}$, we can conclude that for every $N \in \N$, there exists $K_N>0$ such that
$$\|\doublehat{\, u \,}\!(\xi,\eta)\|_{\HS} \leq K_N(\langle \xi \rangle + \langle \eta \rangle)^{-N},
$$
for all $([\xi],[\eta]) \in \widehat{G_1} \times \widehat{G_2}$. Therefore $u \in C^{\infty}(G)$.
\end{proof}
\section{Low order perturbations}

In view of the Greenfield-Wallach conjecture, a way to obtain e\-xam\-ples of globally hypoelliptic first order differential operators defined on compact Lie groups other than the torus is to consider perturbations of vector fields by low order terms.

We start by considering the case where $q$ is a constant, next we will consider perturbations by functions $q \in C^\infty(G)$. This approach was inspired by the reference \cite{Ber94} of A. Bergamasco. In both situations, perturbations by constant and functions, we characterize the global hypoellipticity and the global solvability.

\subsection{Perturbations by constants}\label{addconstant}

Let $G$ be a compact Lie group, $X \in \mathfrak{g}$ and $q \in \C$. Define the operator \newline ${L_q:C^\infty(G) \to C^\infty(G)}$ as
$$
L_q u:= Xu + q u, \quad   u \in C^\infty(G).
$$
We can extend $L_q$ to $\DG$ as
\begin{equation}\label{extend}
\jp{L_q u, \varphi} := -\jp{u, X\varphi} + \jp{u,q \varphi} = - \jp{u,L_{-q} \varphi}, \quad  u \in \DG,\ \varphi \in C^\infty(G).
\end{equation}
If $L_{q}u=f \in C^\infty(G)$, the Fourier coefficient of $f$ can be obtained as 
$$
\widehat{f}(\xi) = \widehat{L_q u}(\xi) = \widehat{Xu}(\xi) + \widehat{q u}(\xi) = \sigma_X(\xi) \widehat{u}(\xi) + q \widehat{u}(\xi),
$$
for all $[\xi] \in \widehat{G}$. So
$$
\widehat{f}(\xi)_{mn}=i\lambda_m(\xi)\widehat{u}(\xi)_{mn} + q \widehat{u}(\xi)_{mn} = i(\lambda_m(\xi) - iq ) \widehat{u}(\xi)_{mn},
$$
for all $[\xi]\in \widehat{G}$, $1 \leq m,n \leq d_\xi$. 

From this we conclude that 
$$
\widehat{f}(\xi)_{mn} = 0, \mbox{ whenever } \lambda_m(\xi) - iq = 0.
$$

In addition, if $\lambda_m(\xi) - iq \neq 0$, then
$$
\widehat{u}(\xi)_{mn} = \frac{1}{i(\lambda_m(\xi) - iq )}\widehat{f}(\xi)_{mn}.
$$

Thus, we obtain the following characterization for the global hypoellipticity and solvability of $L_q$ which is similar to the vector fields case and so its proof will be omitted.
\begin{thm}\label{thmper}
	The operator $L_q=X+q$ is globally hypoelliptic if and only if the following conditions are satisfied: 
	\begin{enumerate}[1.]
		\item The set
		$$
		\mathcal{N}=\{[\xi] \in  \widehat{G} ; \ \lambda_m(\xi)-iq = 0  \mbox{ for some }  1 \leq m \leq d_\xi \}
		$$
		is finite.
		\item $\exists C, \, M>0$ such that
		\begin{equation}\label{hypothesisper}
		|\lambda_m(\xi)-iq|\geq C\jp{\xi}^{-M}, 
		\end{equation}
		for all  $[\xi] \in \widehat{G}, \ 1 \leq m \leq d_\xi$ whenever $\lambda_m(\xi)+iq \neq 0$.
	\end{enumerate}
\end{thm}

	Let $\mathcal{K}_q:=\{w \in \DG; \ {\widehat{w}}(\xi)_{mn}=0 \textrm{ whenever } \lambda_m(\xi)-iq=0, \, \mbox{for all} \, 1\leq m,n \leq d_\xi, 1 \leq r,s \leq d_\eta\}$.

\begin{defi}
 We say that $L_q$ is globally solvable if $L_q(\DG) = \mathcal{K}_q$.
\end{defi}
\begin{thm}\label{thmper2}
	The operator $L_q=X+q$ is globally solvable if and only if the condition \eqref{hypothesisper} is satisfied, that is, 
	$\exists C, \, M>0$ such that
	\begin{equation*}
	|\lambda_m(\xi)-iq|\geq C\jp{\xi}^{-M}, 
	\end{equation*}
	for all  $[\xi] \in \widehat{G}, \ 1 \leq m \leq d_\xi$ whenever $\lambda_m(\xi)+iq \neq 0$.
\end{thm}

\begin{cor}\label{implication}
	If $L_q$ is globally hypoelliptic, then $L_q$ is globally solvable.
\end{cor}

Recall the definition of global hypoellipticity modulo kernel given in Section \ref{weak}. The proof of the next result is similar to Proposition \ref{modker} and its proof will be omitted. 
\begin{prop}
	The operator $L_q$ is globally hypoelliptic modulo $\ker L_q$ if and only if $L_q$ is globally solvable.
\end{prop}
\begin{ex}\label{examples32}
	Let $G=\St$ and consider 
	$$
	L_q=X+q, \quad q \in \C.
	$$
	As in Example \ref{examples3} we can assume that the vector field $X$ defined on $\St$ has the symbol
	$$
	\sigma(\ell)_{mn} = im\delta_{mn}, \quad \ell \in \tfrac{1}{2}\N_0, \ -\ell \leq m,n \leq \ell, \ell-m,\ell-n \in \N_0.
	$$
	In this case,
	$$
	\mathcal{N} = \{\ell \in \tfrac{1}{2}\N_0;  m-iq = 0, \mbox{ for some } -\ell \leq m \leq \ell, \ \ell-m \in \N_0 \}
	$$
	and the inequality \eqref{hypothesisper} becomes
	\begin{equation}\label{condition2s3}
	|m-iq| \geq C(1+\ell)^{-M},
	\end{equation}
	for all $\ell \in \tfrac{1}{2}\N_0$, $-\ell \leq m \leq \ell$, $\ell-m\in\N_0$ whenever $m-iq \neq 0$, for some $C,M >0$.
	
	We have $\mathcal{N} = \varnothing$ for all $q \notin i\tfrac{1}{2}\Z$ and $\mathcal{N}$ has infinitely many elements otherwise. Notice that the condition \eqref{condition2s3} is always satisfied. Therefore $L_q$ is globally hypoelliptic if and only if $q \notin i\tfrac{1}{2}\Z$ and $L_q$ is globally solvable for all $q \in \C$.
\end{ex}

For examples on the torus, we refer the reader to Proposition 4.1 of \cite{Ber94} where the author presented conditions to obtain global hypoellipticity of perturbed operators by constants and constructed some enlightening examples.

\subsection{Perturbations by functions}

In this section we are concerned with the operator $L_q:= X+q$,  considering now $q \in C^\infty(G)$. The idea is to establish a connection between the global hypoellipticity of $L_q$ and $L_{q_0} = X+q_0$, where $q_0$ is the average of $q$ in $G$.
 
 Let $G$ be a compact Lie group, $X \in \mathfrak{g}$, and $Q \in C^\infty(G)$. We can write
$$
XQ = q - q_0,
$$
where $q \in C^\infty(G)$ and $q_0 = \int_G q(x) \, dx$. 
\begin{lemma}\label{solution}
For any $\varphi \in C^\infty(G)$ we have
	$$
	Xe^{\varphi} = (X\varphi)e^{\varphi}.
	$$
\end{lemma}
\begin{proof}
Let $x \in G$, then
\begin{align}
(Xe^\varphi)(x) & = X \sum_{k=0}^\infty \frac{\varphi(x)^k}{k!} = \sum_{k=0}^\infty \frac{X\varphi(x)^k}{k!} = \sum_{k=1}^\infty \frac{k\varphi(x)^{k-1}(X\varphi)(x)}{k!} \nonumber \\
&= (X\varphi)(x) \sum_{k=1}^\infty \frac{\varphi(x)^{k-1}}{(k-1)!} = (X\varphi)(x) e^{\varphi(x)} \nonumber .
\end{align}

\end{proof}

Let $L_q:C^\infty(G) \to C^\infty(G)$ be defined by
$$
L_q u = Xu(x) + q u, \quad  u \in C^\infty(G).
$$
We can extend $L_q$ to $\DG$ as in \eqref{extend}.

\begin{prop}\label{conjugation}
	Assume that there exists $Q \in C^\infty(G)$ such that $XQ = q - q_0$, where $q_0=\int_G q(x) \, dx$. Then 
	\begin{enumerate}[1.]
		\item $L_{q} \circ e^{-Q} = e^{-Q} \circ L_{q_0},$ in both $C^\infty(G)$ and in $\DG$;
		\item $L_{q}$ is globally hypoelliptic if and only if $L_{q_0}$ is globally hypoelliptic;
		\item $L_{q}$ is globally hypoelliptic modulo $\ker L_q$ if and only if $L_{q_0}$ is globally hypoelliptic modulo $\ker L_{q_0}$.
	\end{enumerate}
\end{prop}
\begin{proof}

1. Let $u \in C^\infty(G)$. Then
\begin{align*}
(L_{q} \circ e^{-Q})u &= L_{q}(e^{-Q}u) \\
&= X(e^{-Q}u) + q e^{-Q}u = (Xe^{-Q}) u + e^{-Q}Xu + q e^{-Q} u \nonumber \\
&= (-XQ) e^{-Q} u+ e^{-Q}Xu + q e^{-Q}u  \\
&= -(q-q_0)e^{-Q} u+ e^{-Q}Xu + q e^{-Q}u \nonumber\\
&=  e^{-Q}(Xu+q_0 u) \\
&= (e^{-Q} \circ L_{q_0})u \nonumber.
\end{align*}
The same is true when we have $u\in \DG$.

2. Suppose that $L_q$ is globally hypoelliptic. If $L_{q_0}u = f \in C^{\infty}(G)$ for some $u\in \DG$, then
$e^{-Q}L_{q_0}u = e^{-Q} f \in C^{\infty}(G)$. Since $ e^{-Q} \circ L_{q_0}=L_{q} \circ e^{-Q},$ we have
$L_q(e^{-Q} u) \in C^\infty(G)$ and, by the global hypoellipticity of $L_{q}$ we have $e^{-Q}u \in C^\infty(G)$, which implies that $u\in C^\infty(G)$ and then $L_{q_0}$ is globally hypoelliptic.

Assume now that $L_{q_0}$ is globally hypoelliptic. If $L_q u = f \in C^\infty(G)$ for some $u\in \DG$, we can write $L_q (e^{-Q} e^Q u) = f \in C^\infty(G)$. By the fact that $L_{q} \circ e^{-Q} = e^{-Q} \circ L_{q_0} $ we obtain $e^{-Q} L_{q_0}(e^Q u) = f$, that is, $L_{q_0}(e^Q u) = e^Q f \in C^\infty(G)$. By the global hypoellipticity of $L_{q_0}$ we have that $e^Q u \in C^\infty(G)$ and then $u \in C^\infty(G)$.

3. The proof is analogous to the item 2.
\end{proof}

Now assume that $L_q u=f \in \DG$ for some $u\in\DG$. We may write $u=e^{-Q}(e^Q u)$, so $L_q (e^{-Q}(e^Q u))=f$. By Proposition \ref{conjugation}, we have $e^{-Q} L_{q_0} e^{Q}u = f$, that is,
$$
L_{q_0}e^{Q}u = e^{Q}f.
$$
This implies that $e^Q f \in \mathcal{K}_{q_0}$.

\begin{defi}\label{definitionsolv}
	We say that the operator $L_q$ is globally solvable if:
	\begin{enumerate}[1.]
		\item there is $Q$ such that $XQ = q-q_0$, where  $q_0 = \int_G q(x)\, dx$; and 
		\item $L_{q}(\DG) = \mathcal{J}_{q}$, where
	$$
	\mathcal{J}_q := \{v \in \DG; \ e^Q v \in \mathcal{K}_{q_0} \}.
	$$
\end{enumerate}
\end{defi}

\begin{prop}\label{solvability2}
	$L_q$ is globally solvable if and only if $L_{q_0}$ is globally solvable.
\end{prop}
\begin{proof}
Assume that $L_q$ is globally solvable and let $f \in \mathcal{K}_{q_0}$. Let us show that there exists $u\in \DG$ such that $L_{q_0}u=f$. We can write $f=e^{Q}e^{-Q}f$, so $e^{-Q }f \in \mathcal{J}_{q}$. Since $L_q$ is globally solvable, there exists $v \in \DG$ such that $L_q v = e^{-Q }f $. We can write $v=e^{-Q} e^Q v$ and then $L_q (e^{-Q} e^Q v) = e^{-Q }f $. By Proposition \ref{conjugation}, we have
$$
e^{-Q}L_{q_0}e^Q v = L_q (e^{-Q} e^Q v) = e^{-Q }f,
$$
that is, $L_{q_0}e^Q v=f$.

Suppose now that $L_{q_0}$ is globally solvable and let $f \in \mathcal{J}_q$. By the definition of $\mathcal{J}_q$, we have $e^Q f \in \mathcal{K}_{q_0}$ and by the global solvability of $L_{q_0}$, there exists $u \in \DG$ such that $L_{q_0} u = e^Q f$, that is, $e^{-Q}L_{q_0}u = f$. By Proposition \ref{conjugation}, we get $L_q e^{-Q}u =f$.

\end{proof}

\begin{cor}
	If $L_q$ is globally hypoelliptic then $L_{q}$ is globally solvable.
\end{cor}
\begin{proof}
Suppose that $L_q$ is globally hypoelliptic. By Proposition \ref{conjugation} the operator $L_{q_0}$ is globally hypoelliptic, so by Corollary \ref{implication}, $L_{q_0}$ is globally solvable. Finally, by Proposition \ref{solvability2}, we conclude that $L_q$ is globally solvable.
\end{proof}

\begin{cor}
 $L_q$ is globally hypoelliptic modulo $\ker L_q$ if and only if $L_q$ is globally solvable.
\end{cor}

\begin{ex}
	Let $G=\mathbb{T}^2$ and consider
	$$
	L_q = \partial_t + \partial_x + q(t,x),
	$$
	where $q(t,x) = \sin(t+x)$. For $Q(t,x)=-\tfrac{1}{2}\cos(t+x)$ we have $(\partial_t +\partial_x)Q(t,x) = q(t,x)-q_0$, where $q_0=0$. Since $L_{q_0} = \partial_t+\partial_x$ is not globally hypoelliptic, we conclude by Proposition \ref{conjugation} that $L_q$ is not globally hypoelliptic. On the other hand, the operator $L_{q_0}$ is globally solvable, so that $L_q$ is globally solvable.
	
	For $q(t,x) = \sin(t+x)+1$, we have $q_0=1$ and by Theorem \ref{thmper} we have that $L_{q_0}$ is globally hypoelliptic and then $L_q$ is globally hypoelliptic, which implies that $L_q$ is also globally solvable.
\end{ex}

\begin{ex}\label{examples33}
	Let $G=\mathbb{S}^3$.	We can identify $\St$ with $\emph{\mbox{SU}}(2)$, and the Euler's angle parametrization of $\emph{\mbox{SU}}(2)$ is given by
	$$
	x(\phi,\theta,\psi) = \left(
	\begin{array}{cc}
	\cos\left(\tfrac{\theta}{2}\right) e^{i(\phi+\psi)/2} & i	\sin\left(\tfrac{\theta}{2}\right) e^{i(\phi-\psi)/2}\\
	i	\sin\left(\tfrac{\theta}{2}\right) e^{-i(\phi-\psi)/2} & 	\cos\left(\tfrac{\theta}{2}\right) e^{-i(\phi+\psi)/2}
	\end{array}
	\right)  \in \emph{\mbox{SU}}(2) ,
	$$
	where $0 \leq \phi < 2\pi$, $0 \leq \theta \leq \pi$ and $-2\pi \leq \psi < 2\pi$. Hence, the trace function on $\emph{\mbox{SU}}(2)$ in Euler's angles is given by
	$$
	\emph{\mbox{tr}}(x(\phi,\theta,\psi)) = 2\cos\left(\tfrac{\theta}{2}\right)\cos\left(\tfrac{\phi+\psi}{2}\right).
	$$ 
	
	Consider the operator
	$$
	L_q = X + q(x),
	$$
	where $X$ is the same vector field from Example \ref{examples3} and $q(x)=h(x)+\sqrt{2}$, where $h: \St \to \C$ is expressed in Euler's angles by
	\begin{equation}\label{functions3}
	h(x(\phi,\theta,\psi)) = -\cos\left(\tfrac{\theta}{2}\right)\sin\left(\tfrac{\phi+\psi}{2}\right).
	\end{equation}

	The operator $X$ in Euler's angles is the operator $\partial_{\psi}$ and then we have 
	$$
	X\emph{\mbox{tr}}(x) = q(x) - \sqrt{2}.
	$$ 
	By Example \ref{examples32}, the operator $L_{q_0}=X+\sqrt{2}$ is globally hypoelliptic, then by Proposition \ref{conjugation} we have that $L_q$ is globally hypoelliptic, which implies that $L_q$ is also globally solvable. 

For $q(x)=h(x) + \frac{3}{2}i$ we have by Example \ref{examples3} that the operator $L_{q_0} = X+ \frac{3}{2}i$ is not globally hypoelliptic, then $L_{q}$ is not globally hypoelliptic. On the other hand, the operator $L_{q_0}$ is globally solvable, which implies that $L_{q}$ is globally solvable and globally hypoelliptic modulo kernel.
\end{ex}

\section{A class of vector fields with variable coefficients}\label{general}
Let $G_1$ and $G_2$ be compact Lie groups, and set $G=G_1 \times G_2$. In this section we will characterize the global hypoellipticity and the global solvability for operators in the form
$$
L_{aq}=X_1+ a(x_1)X_2+q(x_1,x_2),
$$
where $X_1 \in \mathfrak{g}_1$, $X_2 \in \mathfrak{g}_2$, $a \in C^\infty(G_1)$ is a real-valued function, and $q \in C^\infty(G)$. First, let us consider the case where $q \equiv 0$.

\subsection{Normal form}

Let
$$
L_a= X_1 + a(x_1)X_2,
$$
where $X_1 \in \mathfrak{g}_1$, $X_2 \in \mathfrak{g}_2$ and $a \in C^\infty(G_1)$ is a real-valued function. If $L_a u=f \in C^\infty(G)$, taking the partial Fourier coefficients with respect to the second variable, we obtain
$$
\widehat{L_a u}(x_1,\eta)_{rs}=X_1\widehat{u}(x_1,\eta)_{rs}+i\mu_r(\eta)a(x_1)\widehat{u}(x_1,\eta)_{rs}=\widehat{f}(x_1,\eta)_{rs},
$$
for all $[\eta]\in \widehat{G}_1$, $1\leq r,s\leq d_\eta$. The idea now is to find $\varphi(\: \cdot \:,\eta)_{rs} \neq 0$ such that
$$
v(x_1,\eta)_{rs}=\varphi(x_1,\eta)_{rs}\widehat{u}(x_1,\eta)_{rs}
$$
satisfies
$$
X_1v(x_1,\eta)_{rs}+i\mu_r(\eta)a_0 v(x_1,\eta)_{rs}=\varphi(x_1,\eta)_{rs}\widehat{f}(x_1,\eta)_{rs} := g(x_1,\eta)_{rs},
$$
for all $[\eta]\in \widehat{G}_1$, $1\leq r,s\leq d_\eta$, for some $a_0 \in \mathbb{R}$. So
\begin{align*}
	\varphi(x,\eta)_{rs}\widehat{f}(x_1,\eta)_{rs} &=X_1(\varphi(x_1,\eta)_{rs} \widehat{u}(x_1,\eta)_{rs}) + i \mu_r(\eta)a_0\varphi(x_1,\eta)_{rs}\widehat{u}(x_1,\eta)_{rs} \nonumber \\
	&= X_1(\varphi(x_1,\eta)_{rs})\widehat{u}(x_1,\eta)_{rs}  \\
	& \quad + \varphi(x_1,\eta)_{rs} (X_1\widehat{u}(x_1,\eta)_{rs})  + i\mu_r(\eta)a_0\varphi(x_1,\eta)_{rs}\widehat{u}(x_1,\eta)_{rs} \nonumber \\
	&=X_1(\varphi(x_1,\eta)_{rs})\widehat{u}(x_1,\eta)_{rs} \\ 
	& \quad + \varphi(x_1,\eta)_{rs}\bigr( (X_1\widehat{u}(x_1,\eta)_{rs}) +i\mu_r(\eta)a(x_1)\widehat{u}(x_1,\eta)_{rs}\bigr) \nonumber \\ 
		&\quad - i\mu_r(\eta) (a(x_1)-a_0)\varphi(x_1,\eta)_{rs}\widehat{u}(x_1,\eta)_{rs} \\
	&= X_1(\varphi(x_1,\eta)_{rs})\widehat{u}(x_1,\eta)_{rs}+ \varphi(x_1,\eta)_{rs}\widehat{f}(x_1,\eta)_{rs} \\
	&\quad - i\mu_r(\eta) (a(x_1)-a_0)\varphi(x_1,\eta)_{rs}\widehat{u}(x_1,\eta)_{rs} .
	\end{align*}
	
Thus, if $\widehat{u}(x_1,\eta)_{rs} \neq0$, we have
\begin{equation}\label{phieq}
X_1\varphi(x_1,\eta)_{rs} = i\mu_r(\eta) (a(x_1)-a_0)\varphi(x_1,\eta)_{rs}.
\end{equation}

Suppose that there exists $A \in C^\infty(G_1)$ such that $$X_1A(x_1) = a(x_1)-a_0.$$
We can assume that $A$ is a real-valued smooth function. Notice that
\begin{align*}
\int_{G_1} X_1 A(x_1) \, dx_1 & = \int_{G_1} \frac{d}{ds} A(x_1 \exp(sX_1))\bigg|_{s=0} \, dx_1 \\& = \frac{d}{ds} \int_{G_1} A(x_1\exp(sX_1))\, dx_1 \bigg|_{s=0} \\& = \frac{d}{ds} \int_{G_1} A(x_1)\, dx_1 \bigg|_{s=0}=0.
\end{align*}
So
$$ 0 = \int_{G_1} X_1A(x_1)\,dx_1 = \int_{G_1}(a(x_1)-a_0)\,dx_1 .$$
Therefore $a_0=\displaystyle\int_{G_1}a(x_1)\, dx_1$ and the equation \eqref{phieq} becomes
\begin{equation}\label{eqphi2}
X_1\varphi(x_1,\eta)_{rs} = i\mu_r(\eta) (X_1A)(x_1)\varphi(x_1,\eta)_{rs},
\end{equation}
and by Lemma \ref{solution}, the function 
$$
\varphi(x_1,\eta)_{rs}=e^{i\mu_r(\eta) A(x_1)}
$$
is a solution of \eqref{eqphi2}.

Define the operator $\Psi_a$ as
\begin{equation}\label{psialpha}
\Psi_a  u(x_1,x_2) := \sum_{[\eta]\in \widehat{G_2}}d_\eta \sum_{r,s = 1}^{d_\eta} e^{i\mu_r(\eta) A(x_1)}\widehat{u}(x_1,\eta)_{rs}\,{\eta_{sr}(x_2)}.
\end{equation}

The next lemma is a technical result necessary to show that the operator $\Psi_a$ is well-defined.

\begin{lemma}
	Let $G$ be a compact group, $f \in C^\infty(G)$, and $z \in \C$ with $|z|\geq 1$. Let $\{Y_1, \cdots, Y_d \}$ be a basis for $\mathfrak{g}$. Then for all $\beta \in \N_0^d$, there exists $C_\beta >0$ such that
	\begin{equation}\label{lemmaexp}
	|\partial^\beta e^{z f(x)}| \leq C_\beta |z|^{|\beta|} e^{\emph{\mbox{Re}}(zf(x))}, \quad \forall x \in G,
	\end{equation}
	with $\partial^\beta$ as in \eqref{derivative}.
\end{lemma}
\begin{proof}
Let us proceed by induction on $|\beta|$.

For $|\beta|=0$, we have
$$
|\partial^\beta e^{zf(x)}| = |e^{zf(x)}| = e^{\emph{\emph{\mbox{Re}}}(zf(x))}.
$$

Suppose now that \eqref{lemmaexp} holds for every $\gamma \in \N_0^d$ with $|\gamma| \leq k$ and let $\beta \in \N_0^d$ with $|\beta|=k+1$. We can write $\beta = \gamma+e_j$, for some $j=1,\cdots,d$ and $|\gamma|=k$. So
\begin{align*}
|\partial^\beta e^{zf(x)}| = |\partial^\gamma Y_je^{zf(x)}| = |\partial^\gamma(zY_jf(x) e^{zf(x)})|
&\leq |z| \sum_{\gamma'+\gamma''=\gamma} |\partial^{\gamma'} Y_jf(x)|\,| \partial^{\gamma''}e^{zf(x)}| \\
&\leq C_\beta |z|^{|\beta|}e^{\emph{\emph{\mbox{Re}}}(zf(x))}.
\end{align*}
\end{proof}
\begin{rem}
We have a similar result for the case where $|z| \leq 1$. In this case, the power of $|z|$ on the estimate \eqref{lemmaexp} is equal to $1$ for every $\beta\in\N_0$, i.e., for all $\beta \in \N_{0}^d$ there exists $C_\beta$ such that
	\begin{equation*}\label{lemmaexp2}
|\partial^\beta e^{z f(x)}| \leq C_\beta |z| e^{\emph{\mbox{Re}}(zf(x))}, \quad \forall x \in G.
\end{equation*}
\end{rem}
\begin{prop}\label{autom}
	The operator $\Psi_a$ defined in \eqref{psialpha} is an automorphism of $C^\infty(G)$ and of $\DG$.
\end{prop}	
\begin{proof}
To demonstrate this proposition we will use results about partial Fourier series developed on \cite{KMR19}.

First of all, notice that $\Psi_{-a}$ is the inverse of $\Psi_a$, therefore we only need to prove that $\Psi_a (C^\infty(G)) = C^{\infty}(G)$ and $\Psi_a (\DG) = \DG$.

Let $\beta \in \N_0$ and $u\in C^\infty(G)$. We will show that $\Psi_a u \in C^\infty(G)$. Notice that $\widehat{\Psi_a u}(x_1,\eta)_{rs} = e^{i\mu_r(\eta)A(x_1)}\widehat{u}(x_1,\eta)_{rs}$ and $\mu_r(\eta) A(x_1) \in \mathbb{R}$, for all $[\eta] \in \widehat{G_2}$, $1 \leq r \leq d_\eta$ and $x_1 \in G_1$. Using \eqref{lemmaexp} we obtain
\begin{align}
|\partial^\beta \widehat{\Psi_a u}(x_1,\eta)_{rs}| &= | \partial^\beta (e^{i\mu_r(\eta)A(x_1)}\widehat{u}(x_1,\eta)_{rs})| \nonumber \\
&= \left|\sum_{\beta'+\beta''=\beta } \partial^{\beta'}e^{i\mu_r(\eta)A(x_1)}\partial^{\beta''}\widehat{u}(x_1,\eta)_{rs} \right|\nonumber \\
&\leq \sum_{\beta'+\beta''=\beta}  \left|\partial^{\beta'} e^{i\mu_r(\eta)A(x_1)}\right| \left|\partial^{\beta''}\widehat{u}(x_1,\eta)_{rs}\right| \nonumber \\
&\stackrel{}{\leq} \sum_{\beta'+\beta''=\beta }C_{\beta'} |\mu_r(\eta)|^{|\beta '|}\left|\partial^{\beta''} \widehat{u}(x_1,\eta)_{rs}\right| \nonumber.
\end{align}
Since $u \in C^\infty(G)$ and $|\mu_r(\eta)|\leq \jp{\eta}$, it is easy to see that given $N>0$, there exists $C_{\beta N}$ such that
$$
|\partial^\beta \widehat{\Psi_a u}(x_1,\eta)_{rs}| \leq C_{\beta N} \jp{\eta}^{-N}.
$$
Therefore $\Psi_a u \in C^\infty(G)$. The distribution case is analogous.
\end{proof}
\begin{prop}\label{conjugation2}
Let $a \in C^\infty(G_1)$,  $a_0:= \int_{G_1} a(x_1) \, dx_1$, and consider the operator
	$
	L_{a_0} = X_1+a_0 X_2.
	$
	Assume that there exists $A \in C^\infty(G_1)$ such that $X_1A = a - a_0$.
	Then we have
	$$
	L_{a_0} \circ \Psi_{a} = \Psi_{a} \circ L_a
	$$
	in both $C^\infty(G)$ and in $\DG$, where $\Psi_a$ is given in \eqref{psialpha}.
\end{prop}
\begin{proof}
Let us show that for any $u\in C^\infty(G)$ we have
$$
\widehat{L_{a_0}(\Psi_a u)}(x_1,\eta)_{rs} = \widehat{\Psi_a(L_a u)}(x_1,\eta)_{rs},$$
for all $x_1 \in G_1, \ [\eta] \in \widehat{G}_2, \ 1 \leq r,s \leq d_\eta.$

Indeed,
\begin{align*}
\widehat{L_{a_0}(\Psi_a u)}(x_1,\eta)_{rs} &= \widehat{X_1 \Psi_a u}(x_1,\eta)_{rs} + a_0\widehat{X_2\Psi_a u}(x_1,\eta)_{rs} \nonumber\\
&= X_1\widehat{\Psi_a u}(x_1,\eta)_{rs} + i\mu_r(\eta)a_0 \widehat{\Psi_a u}(x_1,\eta)_{rs} \nonumber \\
&= X_1(e^{i\mu_r(\eta)A(x_1)}\widehat{u}(x_1,\eta)_{rs}) + i\mu_r(\eta)a_0e^{i\mu_r(\eta)A(x_1)} \widehat{u}(x_1,\eta)_{rs} \\
&= (X_1e^{i\mu_r(\eta)A(x_1)})\widehat{u}(x_1,\eta)_{rs} + e^{i\mu_r(\eta)A(x_1)}(X_1\widehat{u}(x_1,\eta)_{rs}) \nonumber \\ & \quad + i\mu_r(\eta)a_0e^{i\mu_r(\eta)A(x_1)} \widehat{u}(x_1,\eta)_{rs}.
\end{align*}

By \eqref{eqphi2} we have
\begin{align*}
\widehat{L_{a_0}(\Psi_a u)}(x_1,\eta)_{rs} &= i\mu_r(\eta) (a(x_1)-a_0)e^{i\mu_r(\eta)A(x_1)}\widehat{u}(x_1,\eta)_{rs} \\ &\quad  + e^{i\mu_r(\eta)A(x_1)}(X_1\widehat{u}(x_1,\eta)_{rs}) \nonumber \\ &\quad + i\mu_r(\eta)a_0e^{i\mu_r(\eta)A(x_1)} \widehat{u}(x_1,\eta)_{rs} \nonumber \\
&= e^{i\mu_r(\eta)A(x_1)}(X_1\widehat{u}(x_1,\eta)_{rs}+i\mu_r(\eta)a(x_1) \widehat{u}(x_1,\eta)_{rs}) \nonumber\\
&= e^{i\mu_r(\eta)A(x_1)}\widehat{L_{a}u}(x_1,\eta)_{rs} \nonumber \\
&=\widehat{\Psi_a(L_a u)}(x_1,\eta)_{rs} \nonumber .
\end{align*}

The same is true when $u \in \DG$.
\end{proof}

\subsection{Global properties}

Recall that the operator $L_{a_0}$ is globally solvable when ${L_{a_0}(\DG) = \mathcal{K}_{a_0}}$, where
\begin{align*}
  \mathcal{K}_{a_0} := \{w \in \DG; \doublehat{{\, w\,}}(\xi,\eta)_{mn_{rs}} \!= 0  \mbox{ whenever } \lambda_m(\xi) + a_0 \mu_r(\eta) = 0, \\ \mbox{for all} \, 1\leq m,n \leq d_\xi, 1 \leq r,s \leq d_\eta \}.
\end{align*}
We will say that $L_a$ is globally solvable if $L_a(\DG) = \mathcal{J}_{a}$, where
$$
\mathcal{J}_{a} := \{v \in \DG; \Psi_{-a}v \in \mathcal{K}_{a_0} \}.
$$

\begin{prop}\label{ghevolution}
	The operator $L_a$ is globally hypoelliptic (resp. globally hypoelliptic modulo $\ker L_a$) if and only if $L_{a_0}$ is globally hypoelliptic (resp. globally hypoelliptic modulo $\ker L_{a_0}$). Similarly, the operator $L_a$ is globally solvable if and only if $L_{a_0}$ is globally solvable.
\end{prop}
\begin{proof}
The proof is analogous to the demonstration of Propositions \ref{conjugation} and \ref{solvability2}.
\end{proof}

\begin{cor}
	If $L_a$ is globally hypoelliptic, then $L_a$ is globally solvable.
\end{cor}
\begin{cor}
	The operator $L_a$ is globally hypoelliptic modulo $\ker L_{a}$ if and only if $L_a$ is globally solvable.
\end{cor}

\begin{ex}\label{examplet1g2}
	
	Let $G=\mathbb{T}^1 \times G_2$, where $G_2$ is a compact Lie group and $a \in C^\infty(\mathbb{T}^1)$ a real-valued function. Let
	$$L_a := \partial_t + a(t) X_2,$$
	where $X_2 \in \mathfrak{g}_2$, $$a_0= \frac{1}{2\pi} \int_{0}^{2\pi} a(t)\, dt,$$ and define
	$$A(t)=\int_{0}^{t} a(s) \, ds - a_0t.$$
	Notice that $\partial_t A(t) = a(t)-a_0$, for all $t \in \mathbb{T}^1$. Taking the Fourier coefficients with respect to the second variable, we obtain
	$$
	\partial_t\widehat{u}(t,\eta)_{rs}+ia(t)\mu_r(\eta)\widehat{u}(t,\eta)_{rs}=\widehat{f}(t,\eta)_{rs}.
	$$
	
	Define the operator $\Psi_a$ as 
	$$\Psi_a  u(t,x_2) := \sum_{[\eta]\in\widehat{G_2}} d_\eta \sum_{r,s=1}^{d_\eta} e^{i\mu_r(\eta)A(t)}\widehat{u}(t,\eta)_{rs}\, \eta_{sr}(x_2).
	$$
	By Proposition \ref{autom}, the operator $\Psi_a$ is an automorphism of $C^\infty(G)$ and of $\DG$ and it holds
	$$
	\Psi_a\circ L = L_{a_0}\circ \Psi_a.
	$$
	Thus, the operator $L_{a}$ is globally hypoelliptic if and only if $L_{a_0}$ is globally hypoelliptic. Moreover, the operator $L_{a}$ is globally solvable if and only if $L_{a_0}$ is globally solvable. 
	
	For instance, for $a(t)=\sin(t)+\sqrt{2}$, we have $a_0=\sqrt{2}$ and $A(t)=-\cos(t)$. 
	Take $G_2=\mathbb{T}^1$. We know by Example \ref{exampleliouville} that the operator
	$$
	L_{a_0} = \partial_t + \sqrt{2}\partial_x 
	$$
	is globally hypoelliptic, because $\sqrt{2}$ is an irrational non-Liouville number. Hence,
	$$
	L_a = \partial_t + (\sin(t)+\sqrt{2})\partial_x
	$$
	is globally hypoelliptic and then globally solvable. 
	
	Take now $G_2=\St$. By Example \ref{examples3}, we know that
	$$
	L_{a_0} = \partial_t +\sqrt{2}X
	$$
	is not globally hypoelliptic but it is globally solvable. Therefore
	$$
	L_a = \partial_t+(\sin(t)-\sqrt{2})X
	$$
	is not globally hypoelliptic and it is globally solvable. 
\end{ex} 

\begin{rem}
	We had supposed that given a function $a \in C^\infty(G_1)$ there exists a function $A \in C^\infty(G_1)$ and $a_0 \in \mathbb{R}$ such that $X_1 A = a - a_0,$ that is, $X_1$ is $C^\infty$--cohomology free on $G_1$ (see Definition \ref{CH}).
	
	\begin{conjecture}[Katok]
		If a closed, connected, orientable manifold $M$ admits a $C^\infty$--cohomology free vector field $X$, then $M$ is diffeomorphic to a torus and $X$ is smoothly conjugate to a Diophantine vector field.
	\end{conjecture}
	
	In \cite{F08}, G. Forni has proved that the Katok's conjecture is equivalent to the Greenfield's and Wallach's conjecture, mentioned in the previous chapter (see Conjecture \ref{GW}). In view of the proof of this conjecture in dimensions 2 and 3, and its probable validity in higher dimensions, it is necessary to add in the hypothesis the existence of such $A$ satisfying $X_1A=a-a_0$. Otherwise, the above results would be valid only for the case where $G_1$ is a torus.
\end{rem}
\subsection{Perturbations}

We can now combine what was made in Section \ref{addconstant} to study the operator
$$
L_{aq}=X_1+a(x_1)X_2+q(x_1,x_2),
$$
where $X_1 \in \mathfrak{g}_1$, $X_2 \in \mathfrak{g}_2$,  $a \in C^\infty(G_1)$ is a real-valued function, and $q \in C^\infty(G)$. Furthermore, we will assume that there exists $Q \in C^\infty(G)$ satisfying 
$$(X_1+a(x_1)X_2)Q = q - q_0.$$

By Proposition \ref{conjugation} we have
$$
L_{a q } \circ e^{-Q} = e^{-Q}  \circ L_{a q_0},
$$
where $L_{a q_0} = X_1+a(x_1)X_2+q_0$.

It follows from Proposition \ref{conjugation2} that
$$
L_{a_0 q_0} \circ \Psi_a = \Psi_a \circ L_{a q_0},
$$
where $L_{a_0 q_0}=X_1+a_0X_2+q_0$. Thus,
\begin{equation*}
L_{a q } \circ e^{-Q} \circ \Psi_a = e^{-Q} \circ L_{a q_0} \circ \Psi_a= e^{-Q } \circ \Psi_a \circ L_{a_0 q_0}.
\end{equation*}
We say that $L_{aq}$ is globally solvable if $L_{a q} (\DG) = \mathcal{J}_{aq}$, where
$$
\mathcal{J}_{a q} := \{ v\in \DG; \ \Psi_{-a} e^{Q}v \in \mathcal{K}_{a_0q_0} \}
$$
and
\begin{align*}
\mathcal{K}_{a_0 q_0} := \{w \in \DG; \doublehat{\,w\,}(\xi,\eta)_{mn_{rs}}\!=0 \mbox{ whenever }\lambda_m(\xi)+a_0\mu_r(\eta) - iq_0 = 0, \\ \mbox{for all} \, 1\leq m,n \leq d_\xi, 1 \leq r,s \leq d_\eta\}.
\end{align*}
The next results are consequences of what was done previously.

\begin{prop}
	The operator $L_{aq}$ is globally hypoelliptic (resp. globally hypoelliptic modulo $\ker L_{aq}$) if and only if $L_{a_0q_0}$ is globally hypoelliptic (resp. globally hypoelliptic modulo $\ker L_{a_0q_0}$). Similarly, the operator $L_{aq}$ is globally solvable if and only if $L_{a_0q_0}$ is globally solvable.
\end{prop}

\begin{cor}
	If $L_{aq}$ is globally hypoelliptic, then $L_{a_0q_0}$ is globally solvable.
\end{cor}
\begin{cor}
	The operator $L_{aq}$ is hypoelliptic modulo $\ker L_{aq}$ if and only if $L_{a_0q_0}$ is globally solvable.
\end{cor}
\begin{ex}
	Let $G=\mathbb{T}^1 \times \St$ and $X \in \mathfrak{s}^3$ as in the Example \ref{examples3}. Let $a(t)=\sin(t)+\sqrt{2}$ and $q(t,x) = \cos(t)+(\sin(t)+\sqrt{2})h(x)+1$, with $h$ as in Example \ref{examples33}. Here,  $a_0 = \sqrt{2}$ and $q_0 = 1$. Notice that the function $Q(t,x)=\sin(t)+\emph{\mbox{tr}}(x)$ satisfies  $(\partial_t+a(t)X)Q(t,x)= q(t,x)$. By Theorem \ref{thmper} the operator
	$$
	L_{a_0q_0} = \partial_t + \sqrt{2}X + 1
	$$
	is globally hypoelliptic (see Example \ref{examples3}) and then the operator
	$$
	L_{aq} = \partial_t + (\sin(t) + \sqrt{2})X + (\cos(t)+(\sin(t)+\sqrt{2})h(x)+1)
	$$
	is globally hypoelliptic, which implies that $L_{aq}$ is globally solvable.
	
	For $q(t,x)=\cos(t)+(\sin(t)+\sqrt{2})h(x)+i$, the operator 	
	$$
	L_{a_0q_0} = \partial_t + \sqrt{2}X + i
	$$
	is not globally hypoelliptic (see Example \ref{examples32}) and then the operator
	$$
	L_{aq} = \partial_t + (\sin(t) + \sqrt{2})X + (\cos(t)+(\sin(t)+\sqrt{2})h(x)+i)
	$$
	is not globally hypoelliptic. However, since $\sqrt{2}$ is an irrational non-Liouville number, the operator $L_{a_0q_0}$ is globally solvable, which implies that $L_{aq}$ is globally solvable.
\end{ex}
\section*{Acknowledgments}

The authors would like to thank Ricardo Paleari da Silva for comments and suggestions.

\bibliographystyle{amsplain}
\bibliography{biblio}

\end{document}